\documentclass[12pt, reqno]{amsart}
\usepackage{amsfonts}
\usepackage{amsmath,amssymb,latexsym}
\usepackage{tikz-cd}
\usepackage[all,cmtip]{xy}

\newtheorem{theorem}{Theorem}[section]
\newtheorem{proposition}[theorem]{Proposition}
\newtheorem{question}[theorem]{Question}
\newtheorem{definition}[theorem]{Definition}
\newtheorem{conjecture}[theorem]{Conjecture}
\newtheorem{lemma}[theorem]{Lemma}
\newtheorem{corollary}[theorem]{Corollary}
\newtheorem{remark}[theorem]{Remark}
\newtheorem{example}[theorem]{Example}
\newtheorem{${}$}[theorem]{${}$}

\begin{document}

\title{Motivic multiplicativity of complete intersections}

\author{Ze Xu}

\date{July 30, 2026}

\begin{abstract}
For a smooth projective variety endowed with a Chow-Künneth (abbr. CK) decomposition, 
we introduce the notions of motivic multiple twist-multiplicativity and multiplicativity defect to measure the obstruction to the compatibility of multiple intersection products with the given CK decomposition.
These notions extend the more restrictive notion of multiplicativity introduced by Shen--Vial.
We establish their basic properties and derive natural upper bounds for the motivic multiplicativity defects of curves, surfaces, and ample subvarieties of varieties with trivial Chow groups.
We then explicitly determine the motivic 2-fold multiplicativity defect of any smooth Fano or Calabi-Yau complete intersection in a smooth weighted projective space, thereby strengthening a result of Fu in the Calabi-Yau case. In particular, we prove that any smooth Fano or Calabi–Yau hypersurface admits motivic 0-multiplicativity. This generalizes the corresponding result for cubic hypersurfaces, proved independently by Diaz and Fu–Laterveer–Vial, and confirms a conjecture of Voisin in the Calabi–Yau case. As a consequence, certain relative powers of the associated universal families satisfy the Franchetta property. We also obtain several further applications.
\end{abstract}

\maketitle

\section{Introduction}
\bigskip

In his landmark paper \cite{Bei87}, building on Bloch's work \cite{B10}, Beilinson conjectured the existence of a Tannakian category of mixed motives over any field whose derived category contains the category of Chow motives as a full subcategory. Consequently, the Chow motive of any smooth projective variety is expected to carry a canonical ascending filtration and hence to admit many conjugate but distinct CK decompositions \cite{Ja94} lifting the homological Künneth decomposition. This perspective extends Grothendieck's standard conjecture of Künneth type.

Although the general existence conjecture remains far from being resolved, CK decompositions have been constructed for many interesting classes of varieties, including curves, surfaces, abelian varieties, certain hyper-Kähler varieties, certain Calabi--Yau varieties, and ample subvarieties of varieties with trivial Chow groups, as well as products and hyperplane sections of such varieties. On the other hand, whereas the cup product on any Weil cohomology ring of a smooth projective variety is compatible with the cohomological grading, the naive analogue of this classical fact for its Chow motive, equipped with the intersection product, fails in general. It is therefore natural to ask how this failure can be understood and measured.

For simplicity, we work throughout over the field $\mathbb{C}$ of complex numbers. To measure the obstruction to the compatibility of the multiple intersection products on the Chow motive of a smooth projective variety with a given CK decomposition, we introduce, in Definition \ref{def2.3} of Section 2, the notions of motivic multiple twist-multiplicativity and multiplicativity defect, by taking into account the decomposition properties of the small diagonal classes. These notions generalize the more restrictive notion of a multiplicative CK decomposition introduced by Shen and Vial in \cite{SV16a}, since relatively few varieties are expected to satisfy the latter property.

A remarkable phenomenon suggested by a result of Voisin \cite{Voi15}, together with Murre's conjectures (A) and (B) \cite{Mu93}, is that the motivic multiple multiplicativity defects should stabilize as the number of factors increases. This motivates the definition, in Definition \ref{def2.12}, of the stable motivic multiplicativity defect, which in turn leads to vanishing results for modified diagonal classes. After establishing a criterion for motivic multiple twist-multiplicativity, we study the behavior of these notions under several geometric constructions, including products, projective bundles, and blow-ups.

A fundamental problem is to determine the motivic multiple multiplicativity defect of an arbitrary variety. This appears to be a highly challenging task. As a first step, it seems reasonable to formulate the following conjecture.

\begin{conjecture}\label{conj1.1}
Let $X$ be a smooth projective variety of dimension $n$ endowed with a CK decomposition. Then any CK decomposition of $X$ is $m$-fold $mn$-multiplicative for every integer $m\geq 2$. In particular, the motivic $m$-fold multiplicativity defect satisfies
\begin{equation*}
\mathfrak{d}_m(X)\leq mn.
\end{equation*}
\end{conjecture}

This conjecture is, in fact, a straightforward consequence of Murre's conjecture (B). The upper bound above, however, is generally not sharp. Obtaining sharper bounds requires a more refined analysis of the decomposition properties of the small diagonal classes for specific classes of varieties.

The main goal of this article is to verify Conjecture \ref{conj1.1}, or stronger versions of it, for curves, surfaces, and ample subvarieties of varieties with trivial Chow groups, and to determine explicitly the motivic $2$-fold multiplicativity defect of Fano and Calabi--Yau complete intersections in certain ambient varieties.

\begin{theorem}
\leavevmode
\begin{enumerate}
\item[(i)] Conjecture \ref{conj1.1} holds for curves. More precisely, let $X$ be a curve of genus $g\geq 2$. If $2\leq m\leq g$, then $\mathfrak{d}_m(X)\leq m-1$.
If $m\geq g$, then $\mathfrak{d}_m(X)\leq g-1$,
and hence $\mathfrak{sd}(X)\leq g-1$.

\item[(ii)] Let $X$ be a smooth connected projective surface. Then Murre's conjecture (B) holds for a suitable set of CK projectors on $X^{m+1}$ in codimension $2m$ for every integer $m\geq 2$. Consequently,
$\mathfrak{d}_m(X)\leq 2m$.
If, in addition, $X$ is regular, then $\mathfrak{d}_m(X)$ is even and $\mathfrak{d}_m(X)\leq 2m-2$.
Furthermore, if $X$ is swept out by irreducible curves of genus $g$, each supporting a $0$-cycle rationally equivalent to a fixed degree-one $0$-cycle class $o_X$, then $\mathfrak{sd}(X)\leq 6g+2$.

\item[(iii)] Let $Y$ be a smooth connected projective variety with trivial Chow groups, and let $X\subseteq Y$ be a smooth connected ample subvariety of dimension $n$. Then one has $\mathfrak{d}_m(X)\leq mn$
for every integer $m\geq 2$.
\end{enumerate}
\end{theorem}

We next investigate how the geometry of Fano and Calabi--Yau complete intersections in certain varieties with trivial Chow groups constrains their motivic $2$-fold multiplicativity defects. Among other results, we provide a criterion for detecting motivic $0$-multiplicativity for Fano and Calabi--Yau hypersurfaces in such varieties. As a first illustration, we prove the following theorem, which confirms Voisin's Conjecture 3.5 in \cite{Voi12} for Calabi--Yau hypersurfaces.
\begin{theorem}
Any smooth Fano or Calabi--Yau hypersurface in a smooth weighted projective space admits motivic $0$-multiplicativity.
\end{theorem}

Previously, Voisin \cite{Voi15} established a decomposition property for the third small diagonal class of Calabi--Yau hypersurfaces by a very different method and left a related conjecture open. The case of cubic hypersurfaces was established independently in \cite{D21} and \cite{FLV21}, using distinct approaches. From a motivic perspective, Fano and Calabi--Yau hypersurfaces are also particularly interesting, because most of them are not expected to have abelian Chow motives.

For Fano and Calabi--Yau complete intersections, we explicitly determine the motivic $2$-fold multiplicativity defect of a naturally constructed self-dual CK decomposition.

\begin{theorem}
Let $Y:=\mathbb{P}_w^{N}$ be a smooth weighted projective space of dimension $N=n+e$, where $e\geq 1$,
and let $X\subseteq Y$ be a smooth Fano or Calabi-Yau complete intersection of dimension $n$. 
\begin{enumerate}
\item[(i)] There is an equality of cycle classes:
\begin{equation*}
\Delta_{123}^X=\delta_{12*}P(D_1, D_2)+\delta_{13*}P(D_1, D_2)+\delta_{23*}P(D_1, D_2)
+Q(D'_1, D'_2, D'_3)
\end{equation*}
in $\mathrm{CH}^{2n}(X^3)$, where $\Delta_{123}^X$ is the third small diagonal of $X$, $\delta_{ij}: X^2\rightarrow X^3$ are the embeddings onto the big diagonals for $i<j$, and $P(s_1,s_2)\in\mathbb{Q}[s_1,s_2]$, $Q(t_1,t_2,t_3)\in\mathbb{Q}[t_1,t_2,t_3]$ are symmetric homogeneous polynomials.
Here, $D_i=p_i^*D$ and $D'_j=q_j^*D$, where $p_i$ and $q_j$ are the natural projections.

\item[(ii)] The naturally constructed self-dual CK decomposition of $X$ is multiplicative 
if and only if the cycle class $\delta_{X*}D$ is completely decomposable, where $\delta_X: X\rightarrow X^2$ is the diagonal morphism.

\item[(iii)] The motivic 2-fold multiplicativity defect of the natural CK decomposition of $X$ takes values in $\{2k|0\leq k\leq n-1,k\neq\frac{n}{2}\}$. More precisely, it is equal to $2k$ if and only if the cycle class $\delta_{X*}D^{k+1}$ is completely decomposable, whereas $\delta_{X*}D^{k}$ is not.
\end{enumerate}
\end{theorem} 
For general Calabi--Yau complete intersections, part (i) strengthens a main result of Fu in \cite{F13}, which was obtained by a substantially different method.

Under the Fano or Calabi--Yau hypothesis, a key new geometric input is the existence of \emph{isogenous} correspondences and related varieties. Our method focuses on establishing the required relations among cycle classes by studying proper intersections of closely related varieties. This highlights the importance of the geometry of isogenous correspondences on these varieties. Such correspondences may be viewed as analogues of isogenies between abelian varieties, as suggested by Voisin \cite{Voi04}. We also point out that our approach develops ideas originating in the study of K3 surfaces in \cite{Ba19}.

We derive several applications throughout the article. We conclude the introduction by recording the following two consequences.

\begin{corollary}
Let $f:  \mathcal{X}\rightarrow B$ be the universal family of smooth Fano or Calabi-Yau hypersurfaces in a smooth weighted projective space. 
Then, for every positive integer $k\leq 3$, the $k$-th relative power $f_k: \mathcal{X}_{/B}^k\rightarrow B$ of $f$
satisfies the Franchetta property. More precisely, if $\Gamma\in\mathrm{CH}^*(\mathcal{X}_{/B}^k)$ restricts to a homologically trivial cycle class on $\mathcal{X}_b^k$ for a very general closed point $b\in B$,
then $\Gamma|_{\mathcal{X}_b^k}=0$ in $\mathrm{CH}^*(\mathcal{X}_b^k)$ for any closed point $b\in B$.
\end{corollary}

\begin{corollary}
Let $X$ and $B$ be smooth varieties, and let $f: X\rightarrow B$ be a smooth projective morphism. 
Assume that a very general fiber of $f$ admits a multiplicative CK decomposition. 
Then, after replacing $B$ by a dense Zariski open subset $U\subseteq B$, there exists a multiplicative decomposition isomorphism
\begin{equation*}
\mathbf{R}(f|_U)_*\mathbb{Q}\cong\bigoplus_i\mathbf{R}^i(f|_U)_*\mathbb{Q}[-i]
\end{equation*}
in the derived category of sheaves of $\mathbb{Q}$-vector spaces on $U$.

In particular, the universal family of Fano or Calabi-Yau hypersurfaces in a smooth weighted projective space satisfies this property.
\end{corollary}
\bigskip

\textbf{Conventions}. Unless otherwise stated, we work throughout over the field $\mathbb{C}$ of complex numbers.
By a \emph{variety}, we mean a reduced separated scheme of finite type over $\mathbb{C}$. Singular (co)homology groups, Chow groups, Chow rings, and Chow motives are all taken with rational coefficients.
We denote by $\mathfrak{h}: \mathcal{V}(\mathbb{C})\rightarrow\mathrm{CH}\mathcal{M}$ the contravariant functor from the category of smooth projective complex varieties to the rigid pseudo-abelian tensor category of Chow motives with rational coefficients.
\bigskip

\bigskip

\section{Motivic multiple twist-multiplicativity}
\bigskip

In this section, we introduce the notions of motivic multiple twist-multiplicativity and multiplicativity defect,
study their basic properties, formulate several general questions, and examine the first examples, including curves and surfaces.

Let $X$ be a smooth connected projective variety of dimension $n$. 
Assume that Murre's conjecture (A) holds for $X$, that is,
$X$ admits a set $\{\pi_i^X\}_{0\leq i\leq 2n}$ of CK projectors.
Thus, $\pi_i^X$ are mutually orthogonal idempotents lifting the homological Künneth projectors and $\Delta_X=\sum_{i=0}^{2n}\pi_i^X$ in the correspondence ring $\mathrm{CH}^n(X\times X)$. 
For the full statement of Murre's conjectures, 
we refer to Section 1.4 of \cite{Mu93} or Conjecture 5.1 of \cite{Ja94}.
The CK decomposition defines an ascending filtration on the Chow motive $\mathfrak{h}(X)$ of $X$ by
\begin{equation}\label{1}
\mathfrak{h}^{\leq k}(X):=\bigoplus_{i\leq k}\mathfrak{h}^i(X).
\end{equation}
By the Bloch--Beilinson philosophy \cite{Bei87}, the filtration \eqref{1} should be compatible with every $m$-fold intersection product, where $m\geq 2$. In other words, the following diagram in the tensor category of Chow motives should be commutative:
\begin{equation}\label{2}
\begin{tikzcd}
\mathfrak{h}^{\leq i_1}(X)\otimes\cdots\otimes\mathfrak{h}^{\leq i_m}(X) \arrow{r}{\Delta_{I}^X} \arrow{d}{}
& \mathfrak{h}(X) \\
\mathfrak{h}^{\leq \sum_{\ell}i_{\ell}}(X).\arrow[hook]{ru}
\end{tikzcd}
\end{equation}
Here $m\geq 2$, $I=I_m:=\{1, 2, \cdots, m+1\}$ and $\Delta_{I}^X\in\mathrm{CH}^{mn}(X^{m+1})$ denotes the $(m+1)$-th small diagonal class of $X$. More concretely, one expects the following.

\begin{conjecture}(\cite{Bei87})\label{conj2.1}
Fix an integer $m\geq 2$, and suppose that $X$ admits a set $\{\pi_i^X\}_{0\leq i\leq 2n}$ of CK projectors.
Then, for all integers $i_1,\ldots,i_m$ and every integer $k>\sum_{\ell=1}^{m}i_{\ell}$, one has

\begin{equation}\label{3}
\pi_k^X\circ\Delta_{I}^X\circ(\pi_{i_1}^X\otimes\cdots\otimes\pi_{i_m}^X)
=(^t\pi_{i_1}^X\otimes\cdots\otimes{^t}\pi_{i_m}^X\otimes\pi_k^X)_*\Delta_{I}^X=0  
\end{equation}
in $\mathrm{CH}^{mn}(X^{(m+1)n})$.
\end{conjecture}

\begin{remark}
\leavevmode
\begin{enumerate}
\item[(i)] If $X$ admits a set $\{\pi_i^X\}_{0\leq i\leq 2n}$ of CK projectors,
then it also admits a self-dual one, namely one satisfying
${}^t\pi_i^X=\pi_{2n-i}^X$.
Indeed, this follows by replacing $\pi_j^X$ with ${}^t\pi_{2n-j}^X$ for every integer $j>n$.

\item[(ii)] Suppose that the CK decomposition is self-dual, and let
\begin{equation*}
\pi_\ell^{X^{m+1}}:=\sum_{i_1+\cdots+i_{m+1}=\ell}\pi_{i_1}^X\otimes\cdots\otimes\pi_{i_{m+1}}^X
\end{equation*}
denote the induced product CK projectors on $X^{m+1}$. Then Conjecture \ref{conj2.1} is equivalent to the assertion that $\Delta_I^X$ has no component in
\begin{equation*}
\mathrm{CH}_s^{mn}(X^{m+1}):=(\pi_{2mn-s}^{X^{m+1}})_*\mathrm{CH}^{mn}(X^{m+1})
\end{equation*}
for every integer $s<0$. This is a consequence of Murre's conjecture (B) for $X^{m+1}$ in codimension $mn$. 

\item[(iii)] Write $\mathrm{CH}_s^r(X):=(\pi_{2r-s}^X)_*\mathrm{CH}^r(X)$.
Conjecture \ref{conj2.1} implies that the $m$-fold intersection product induces a map
\begin{equation*}
\mathrm{CH}_{s_1}^{r_1}(X)\otimes\cdots\otimes
\mathrm{CH}_{s_m}^{r_m}(X)
\stackrel{\bullet^m}{\longrightarrow}
\mathrm{CH}_{\geq\sum_{\ell}s_\ell}^{\sum_{\ell}r_\ell}(X).
\end{equation*}
Roughly speaking, this means that the lower grading of the Chow groups can only increase under multiple intersection products.
\end{enumerate}
\end{remark}

In \cite{SV16a}, Shen and Vial introduced the notion of a multiplicative CK decomposition for smooth projective varieties. As demonstrated by the work of several authors, this property is rather restrictive and is expected to hold only for very special classes of varieties.

To extend this framework to arbitrary smooth projective varieties, it is fundamental to observe that the decomposition properties of all small diagonal classes govern the multiple intersection products on the Chow motive of a smooth projective variety. Moreover, these decomposition properties are closely related, and their simultaneous consideration is necessary in order to obtain a more complete picture. Motivated by these observations, we introduce the following notions.

\begin{definition}\label{def2.3}
Fix an integer $m\geq 2$, and put $I:=I_m=\{1, 2,\cdots, m+1\}$.
Let $\tau_m\geq 0$ be an integer.

\leavevmode
\begin{enumerate}

\item[(i)] A set $\{\pi_i^X\}_{0\leq i\leq 2n}$ of CK projectors of $X$ is said to be \emph{$m$-fold (twist-)$\tau_m$-multiplicative}, if the $(m+1)$-th small diagonal class of $X$ has a decomposition:
\begin{equation}\label{4}
\Delta_{I}^X=\sum_{k=0}^{2n}\sum_{-\tau_m\leq k-\sum_{\ell}i_{\ell}\leq 0}\pi_k^X\circ\Delta_{I}^X\circ(\pi_{i_1}^X\otimes\pi_{i_2}^X\otimes\cdots\otimes\pi_{i_m}^X)
\end{equation}
in $\mathrm{CH}^{mn}(X^{m+1})$.
Equivalently, for every integer $k>\sum_{\ell=1}^{m}i_{\ell}$ or $k<\sum_{\ell=1}^{m}i_{\ell}-\tau_m$, one has
\begin{equation}\label{5}
\pi_k^X\circ\Delta_{I}^X\circ(\pi_{i_1}^X\otimes\pi_{i_2}^X\otimes\cdots\otimes\pi_{i_m}^X)
=(^t\pi_{i_1}^X\otimes{^t}\pi_{i_2}^X\otimes\cdots\otimes{^t}\pi_{i_m}^X\otimes\pi_k^X)_*\Delta_{I}^X=0.
\end{equation}

\item[(ii)] The variety $X$ is said to be \emph{motivic $m$-fold (twist)-$\tau_m$-multiplicative}, if it admits a $m$-fold $\tau_m$-multiplicative CK decomposition.
We shall also say that $X$ admits motivic $m$-fold (twist-)$\tau_m$-multiplicativity, or simply motivic multiple twist-multiplicativity.

\item[(iii)] The \emph{motivic $m$-fold multiplicativity defect} $\mathfrak{d}_m(X)$ of $X$ is defined to be the smallest integer $\tau_m$ for which $X$ is motivic $m$-fold $\tau_m$-multiplicative. Equivalently, $X$ is strictly motivic $m$-fold $\mathfrak{d}_m(X)$-multiplicative but is not motivic $m$-fold $\tau$-multiplicative for any integer $\tau<\mathfrak{d}_m(X)$.
One may similarly define the $m$-fold multiplicativity defect of a fixed CK decomposition; this invariant reflects the properties of the chosen decomposition rather than those of the variety itself.
By associativity of the intersection product, one has $\mathfrak{d}_m(X)\leq (m-1)\mathfrak{d}_2(X)$.
In particular,
\begin{equation*}
\Delta_{I}\in\bigoplus_{i=0}^{m\mathfrak{d}_2(X)}\mathrm{CH}_{i}^{nm}(X^{m+1}).
\end{equation*}

\item[(iv)] A set $\{\pi_i^X\}_{0\leq i\leq 2n}$ of CK projectors is said to be \emph{weakly} $m$-fold $\tau$-multiplicative 
if, for all $z_{\ell}\in\mathrm{CH}_{s_{\ell}}^{r_{\ell}}(X)$, $1\leq\ell\leq m$, one has
\begin{equation*}
z_1\cdots z_m\in\bigoplus_{s=0}^{\tau}\mathrm{CH}_{\sum_{\ell=1}^{m}s_{\ell}+\tau}^{\sum_{\ell=1}^{m}r_{\ell}}(X).
\end{equation*}
\end{enumerate}
\end{definition}

\begin{remark}
\leavevmode
\begin{enumerate}
\item[(i)] The notions corresponding to different values of $m$ are indeed distinct, as follows from the criterion shown in Proposition \ref{prop2.15}. The case $m=2$ is particularly important. Unless otherwise specified, we shall therefore omit the term ``$2$-fold'' from the corresponding terminology. In particular, we write $\mathfrak{d}(X)=\mathfrak{d}_2(X)$.

\item[(ii)] By definition, motivic $2$-fold $0$-multiplicativity is equivalent to the existence of a multiplicative CK decomposition in the sense of \cite{SV16a}.

\item[(iii)] Roughly speaking, the motivic $m$-fold multiplicativity defect measures how far a variety is from admitting an $m$-fold multiplicative CK decomposition. The vanishing of this defect is precisely motivic $m$-fold $0$-multiplicativity.

\item[(iv)] If $X$ admits a CK decomposition, then obviously $\mathfrak{d}_m(X)\leq 2mn$.
In particular, the motivic $m$-fold multiplicativity defect of $X$ is well-defined.

\item[(v)] Analogous notions may be defined modulo any adequate equivalence relation finer than homological equivalence, such as algebraic equivalence. These notions can also be extended to smooth projective families.

\item[(vi)] Assume that $X$ admits a set $\{\pi_i^X\}_{0\leq i\leq 2n}$ of CK projectors.
Let 
\begin{equation}\label{6}
\widetilde{\Delta}_{123}^X:=\Delta_{123}^X-\sum_{k=0}^{2n}\sum_{k\neq i+j}\pi_k^X\circ\Delta_{123}^X\circ(\pi_{i}^X\otimes\pi_{j}^X).
\end{equation}
Define the associated \emph{modified} intersection product: 
\begin{equation*}
\mathfrak{h}(X)\otimes\mathfrak{h}(X)\stackrel{\widetilde{\Delta}_{123}^X}\longrightarrow\mathfrak{h}(X).
\end{equation*}
Then it is clear that
\begin{equation}\label{7}
\widetilde{\Delta}_{123}^X
=\sum_{k=0}^{2n}\sum_{i+j=k}\pi_k^X\circ\widetilde{\Delta}_{123}^X\circ(\pi_{i}^X\otimes\pi_{j}^X).
\end{equation}
Thus, every CK decomposition is multiplicative with respect to the associated modified intersection product.
\end{enumerate}
\end{remark}

It is natural to formulate the following general questions.

\begin{question}\label{ques2.5}
\leavevmode
\begin{enumerate}
\item[(i)] Let $X$ be a smooth connected projective variety. Determine $\mathfrak{d}_m(X)$, whenever it is well-defined, for every integer $m\geq 2$.

\item[(ii)] Given an integer $\mathfrak{d}\geq 0$, classify the smooth connected projective varieties with motivic multiplicativity defect $\mathfrak{d}$.
\end{enumerate}
\end{question}

A first step toward answering Question \ref{ques2.5} is to obtain upper bounds for the motivic multiple multiplicativity defect. We are led to the following prediction.
\begin{conjecture}\label{conj2.6}
Let $X$ be a smooth connected projective variety of dimension $n$ endowed with a CK decomposition.
Then any CK decomposition of $X$ is $m$-fold $mn$-multiplicative for every integer $m\geq 2$.
In particular, $\mathfrak{d}_m(X)\leq mn$. 
\end{conjecture}

This conjecture is a direct consequence of Murre's conjecture (B).
Indeed, let $\{\pi_i\}_{0\leq i\leq 2n}$ be a set of CK projectors of $X$.
If $k>\sum_{\ell=1}^{m}i_{\ell}$ or $k<\sum_{\ell=1}^{m}i_{\ell}-mn$, 
then one has $2mn+k-\sum_{\ell=1}^{m}i_{\ell}>2nm$ or $2mn+k-\sum_{\ell=1}^{m}i_{\ell}<nm$.
Murre's conjecture (B) for $X^{m+1}$ in codimension $mn$ therefore implies that
\begin{equation*}
\pi_k\circ\Delta_I\circ(\pi_{i_1}\otimes\cdots\otimes\pi_{i_m})
=(\pi_{2n-i_1}\otimes\cdots\otimes\pi_{2n-i_m}\otimes\pi_k)_*\Delta_I=0.
\end{equation*}
\begin{remark}
\leavevmode
\begin{enumerate}
\item[(i)] Even in the case $m=2$, a variety may admit a strictly $2n$-multiplicative CK decomposition. Indeed, let $X$ be a curve of genus $g\geq 1$, and choose two points $a_1,a_2\in X$ which are not rationally equivalent. 
    Set 
    \begin{equation*}
o_i:=[a_i],\quad \pi_0:=o_1\times X,\quad \pi_2:=X\times o_2,\quad \pi_1:=\Delta_X-\pi_0-\pi_2.
    \end{equation*}
Then $\{\pi_i\}_{0\leq i\leq 2}$ forms a set of CK projectors of $X$.
Moreover,
\begin{equation*}
\pi_0\circ\Delta_{I}\circ(\pi_1\otimes\pi_1)=(\pi_1\otimes\pi_1\otimes\pi_0)_*\Delta_{I}
=(o_1-o_2)\times(o_1-o_2)\times X\neq 0.
\end{equation*}
It follows from Theorem \ref{thm2.24} that this CK decomposition is strictly $2$-multiplicative. This example suggests that, in order to attain $\mathfrak{d}_m(X)$, one should choose a more natural CK decomposition, presumably at least a self-dual one.

\item[(ii)] The inequality $\mathfrak{d}_m(X)\leq mn$ is clearly not sharp in general.
For fixed dimension $n\geq 2$, it may be close to be sharp for all smooth projective varieties only when $m$ is small relative to $n$. Smaller bounds can be obtained for special classes of varieties, such as complete intersections.

\item[(iii)] For a fixed variety, we shall see that the sequence $\mathfrak{d}_m(X)$ stabilizes for sufficiently large $m$ under suitable assumptions.
\end{enumerate}
\end{remark}

The notions of motivic $m$-fold twist-multiplicativity for different values of $m$ are closely related.
\begin{proposition}
Suppose that, for some integer $r\geq 3$, the variety $X$ admits an $r$-fold $\tau_r$-multiplicative self-dual CK decomposition satisfying $\pi_{0}^X=o_X\times X$ for some degree-one 0-cycle class $o_X$ on $X$.
Then, for every integer $m$ with $2\leq m\leq r$, the same CK decomposition is $m$-fold $\tau_{r}$-multiplicative.
\end{proposition}
\begin{proof}
By induction on $r-\ell$, it suffices to treat the case $\ell=r-1$. Set $J:=I_{r-1}=\{1, 2, \cdots, r\}$.
By assumption, whenever $k>\sum_{j=2}^{r}i_j$ or $k<\sum_{j=2}^{r}i_j-\tau_r$, 
we have
\begin{eqnarray*}
0=\pi_k^X\circ\Delta_{I}^X\circ(\pi_{0}^X\otimes\pi_{i_2}^X\otimes\cdots\otimes\pi_{i_r}^X) &=& ({}^t\pi_{0}^X\otimes{{}^t}\pi_{i_2}^X\otimes\cdots\otimes{{}^t}\pi_{i_r}^X\otimes\pi_k^X)_*\Delta_{I}^X \\
   &=& (\pi_{2n}^X\otimes\pi_{2n-i_2}^X\otimes\cdots\otimes\pi_{2n-i_r}^X\otimes\pi_k^X)_*\Delta_{I}^X \\
   &=& o_X\times (\pi_{2n-i_2}^X\otimes\cdots\otimes\pi_{2n-i_r}^X\otimes\pi_k^X)_*\Delta_{J}^X.
\end{eqnarray*}
It follows that 
\begin{equation*}
\pi_k^X\circ\Delta_{J}^X\circ(\pi_{i_2}^X\otimes\cdots\otimes\pi_{i_r}^X)
=(\pi_{2n-i_2}^X\otimes\cdots\otimes\pi_{2n-i_r}^X\otimes\pi_k^X)_*\Delta_{J}^X=0.
\end{equation*}
Hence the CK decomposition is $(r-1)$-fold $\tau_r$-multiplicative.
\end{proof}

\begin{remark}
As a consequence, the motivic $m$-fold multiplicativity defect $\mathfrak{d}_m(X)$ is a nondecreasing function of $m$.
\end{remark}

A remarkable phenomenon is that a result of Voisin \cite{Voi15}, combined with Murre's conjectures (A) and (B), implies that the motivic $m$-fold multiplicativity defect $\mathfrak{d}_m(X)$ stabilizes as $m$ tends to infinity.
\begin{proposition}\label{prop2.10}
Assume that $X$ admits a set $\{\pi_i^X\}_{0\leq i\leq 2n}$ of self-dual CK projectors such that $\pi_{0}^X=o_X\times X$ for some degree-one 0-cycle class $o_X$.
Suppose that Murre's conjecture (B) holds for the product CK projectors of $X^{l+1}$ 
in codimension $ln$ for every integer $l\geq 2$.
Then there exist integers $m\geq 2$ and $\tau\geq 0$ such that, for every integer $k\geq m$, the CK decomposition is $k$-fold $\tau$-multiplicative. In particular, $\mathfrak{d}_k(X)\leq\tau$.
\end{proposition}
\begin{proof}
Observe  that for integers $i\geq 2$ and $j\geq 1$, one has
\begin{equation*}
\Delta_{I_{i-1}}\times\underbrace{o_X\times\cdots\times o_X}_j
=(\Delta_{X^{i}}\otimes\underbrace{\pi_{2n}^X\otimes\cdots\otimes\pi_{2n}^X}_j)_*\Delta_{I_{i+j-1}}
\in\mathrm{CH}^{n(i+j-1)}(X^{i+j}).
\end{equation*}
Now if $\Delta_{I_{i-1}}\in\bigoplus_{s=0}^{\ell}\text{CH}_s^{n(i-1)}(X^{i})$ for some integer $\ell\geq 0$, then 
\begin{equation*}
\Delta_{I_{i-1}}\times\underbrace{o_X\times\cdots\times o_X}_j\in\bigoplus_{s=0}^{\ell}\mathrm{CH}_s^{n(i+j-1)}(X^{i+j}).
\end{equation*}
By symmetry, the same conclusion holds for all permutations of this cycle class.

Denote the $k$-th modified small diagonal class of $(X, o_X)$ by
\begin{equation*}
\Gamma^k(X, o_X):=\sum_{J\subseteq\{1,\ldots,k\}, |J|=j<k}(-1)^jp_{J}^*(o_X^{*j})\cdot p_{J'}^*\Delta_{I_{m-j-1}}\in\mathrm{CH}_n(X^k),
\end{equation*}
where $\{1,\cdots,k\}$ is the disjoint union of $J$ and $J'$, and $p_J: X^k\rightarrow X^j$, $p_{J'}: X^k\rightarrow X^{k-j}$ are natural projections onto the factors indexed by $J$ and $J'$, respectively.
Here $o_X^{*j}=p_1^*o_X\cdots p_j^*o_X$.
Invoking Corollary 1.6 in \cite{Voi15}, 
there exists an integer $m\geq 2$ such that $\Gamma^k(X, o_X)=0$ for every integer $k>m$. 
Murre's conjecture (B), applied to $X^m$ and the cycle class $\Delta_{I_{m-1}}$, implies that 
\begin{equation*}
\Delta_{I_{m-1}}\in\bigoplus_{s=0}^{\tau}\mathrm{CH}_s^{nm}(X^{m})
\end{equation*}
for some integer $0\leq\tau\leq nm$. 
The vanishing of the modified diagonal classes then gives
\begin{equation*}
\Delta_{I_{k-1}}\in\bigoplus_{s=0}^{\tau}\mathrm{CH}_s^{n(k-1)}(X^{k}).
\end{equation*}
for every integer $k\geq m$.
Hence the CK decomposition is $k$-fold $\tau$-multiplicative for every integer $k\geq m$.
\end{proof}
\begin{remark}\label{rem2.11}
By Theorem 1.7 of \cite{Voi15}, if $X$ is swept out by irreducible curves of genus $g$, each supporting a 0-cycle rationally equivalent to $o_X$, then $\Gamma^k(X, o_X)=0$ for every integer $k\geq (n+1)(g+1)$.
In Proposition \ref{prop2.10}, one may therefore take $m\leq (n+1)(g+1)-1$.
Assuming Conjecture \ref{conj2.6}, the preceding argument yields $\mathfrak{d}_k\leq n[(n+1)(g+1)-2]$ for every integer $k\geq 2$. This upper bound is generally not expected to be sharp, except perhaps in a few special cases.
\end{remark}

We may now introduce the following invariant.
\begin{definition}\label{def2.12}
Assume that $X$ admits a CK decomposition.
The \emph{stable motivic multiplicativity defect} of $X$ is defined by
\begin{equation*}
\mathfrak{s}\mathfrak{d}(X):=\mathrm{max}\{\mathfrak{d}_m(X)|m\geq 2\}.
\end{equation*}
\end{definition}
\begin{remark}\label{rem2.13}
\leavevmode
\begin{enumerate}
\item[(i)] Whenever it is well defined, one has $\mathfrak{d}(X)\leq\mathfrak{sd}(X)$. 
Moreover, if Murre's conjecture (B) holds for $X^{k+1}$ in codimension $kn$ for every integer $k\geq 2$,
then Remark \ref{rem2.11} gives
\begin{equation*}
\mathfrak{sd}(X)\leq n[(n+1)(g+1)-2].
\end{equation*}

\item[(ii)] If $\mathfrak{d}(X)=0$, then $\mathfrak{sd}(X)=0$.

\item[(iii)] One may similarly define the stable multiplicativity defect of a fixed CK decomposition.

\item[(iv)] It remains unclear how to determine $\mathfrak{sd}(X)$, or even a sharp upper bound for it, for an arbitrary smooth connected projective variety $X$.
\end{enumerate}
\end{remark}

For certain classes of varieties, the stable motivic multiplicativity defect can be bounded. The following statement is a consequence of Voisin's results in \cite{Voi15}.
\begin{proposition}\label{prop2.14}
Let $X$ be a smooth projective rationally connected variety of dimension $n$ admitting a CK decomposition. 
Then $\mathfrak{sd}(X)\leq 2n(n-1)$.
If, in addition, Murre's conjecture (B) holds for $X^n$, then $\mathfrak{sd}(X)\leq n(n-1)$.
\end{proposition}
\begin{proof}
Since $\text{CH}^n(X)=\mathbb{Q}$, any point of $X$ may be chosen to represent $o_X$. By Corollary 3.2 of \cite{Voi15}, the modified diagonal class satisfies $\Gamma^{k}(X, o_X)=0$ for every integer $k>n$.
Since $\Delta_{I_{n-1}}\in\mathrm{CH}^{n(n-1)}(X^n)$, it follows that $\mathfrak{sd}(X)\leq 2n(n-1)$.
If Murre's conjecture (B) holds for $X^n$, then $\Delta_{I_{n-1}}\in\bigoplus_{s=0}^{n(n-1)}\mathrm{CH}_s^{n(n-1)}(X^n)$
and hence $\mathfrak{sd}(X)\leq n(n-1)$.
\end{proof}

Now we give a simple criterion for detecting motivic multiple twist-multiplicativity, generalizing Proposition 8.4 of \cite{SV16a}.
\begin{proposition}\label{prop2.15}
Fix an integer $m\geq 2$ and put $I:=I_m=\{1, 2,\cdots, m+1\}$.
Let $\tau_m\geq 0$ be an integer.
Assume that $X$ admits a set $\{\pi_i^X\}_{0\leq i\leq 2n}$ of self-dual CK projectors.
For integers $\ell, r, s$, set 
\begin{equation*}
\pi_{\ell}^{X^{m+1}}:=\sum_{i_1+\cdots+i_{m+1}=\ell}\pi_{i_1}^X\otimes\cdots\otimes\pi_{i_{m+1}}^X,\quad
\mathrm{CH}_s^{r}(X^{m+1}):=(\pi_{2r-s}^{X^{m+1}})_*\mathrm{CH}^r(X^{m+1}).
\end{equation*}
Write uniquely
\begin{equation}\label{8}
\Delta_{I}^X=\sum_{s=-2n}^{2mn}\Delta_{I}^s, \quad \Delta_{I}^s:=(\pi_{2mn-s}^{X^{m+1}})_*\Delta_{I}^X\in\mathrm{CH}_s^{mn}(X^{m+1}).
\end{equation}
Then the following statements are equivalent:
\begin{enumerate}
\item[(i)] The CK decomposition of $X$ is $m$-fold $\tau_m$-multiplicative.

\item[(ii)] The cycle class $\Delta_{I}^s$ vanishes for every integer $s<0$ or $s>\tau_m$.

\item[(iii)] There is an equality of cycle classes: 
\begin{equation*}
\Delta_{I}=\sum_{s=0}^{\tau_m}\Delta_{I}^s.
\end{equation*}
\end{enumerate}
\end{proposition}
\begin{proof}
Since $^t\pi_i^X=\pi_{2n-i}^X$, one has
\begin{equation*}
\pi_k^X\circ\Delta_{I}^X\circ(\pi_{i_1}^X\otimes\cdots\otimes\pi_{i_m}^X)
=(\pi_{2n-i_1}^X\otimes\cdots\otimes\pi_{2n-i_m}^X\otimes\pi_k^X)_*\Delta_{I}^X
\in\mathrm{CH}_{\sum_{\ell}i_{\ell}-k}^{mn}(X^{m+1}).
\end{equation*}
It remains only to observe that $k>\sum_{\ell}i_{\ell}$ or $k<\sum_{\ell}i_{\ell}-\tau_m$ is equivalent to that $\sum_{\ell}i_{\ell}-k<0$ or $\sum_{\ell}i_{\ell}-k>\tau_m$.
\end{proof}
\begin{remark}
In many situations, including the case of complete intersections, motivic multiple twist-multiplicativity can be readily shown to be preserved under specialization.
\end{remark}

We next establish some fundamental properties of motivic multiple twist-multiplicativity under several basic geometric constructions, including products, projective bundles, and blow-ups.

\begin{proposition}\label{prop2.17}
Let $X$ (resp. $X'$) be a smooth connected projective variety of dimension $n$ (resp. $n'$) admitting an $m$-fold $\tau_m$-multiplicative  (resp. $\tau'_m$-multiplicative) CK decomposition $\{\pi_i^X\}_{0\leq i\leq 2n}$ (resp. $\{\pi_j^{X'}\}_{0\leq j\leq 2n'}$). 
Then the product CK decomposition of $X\times X'$ is $m$-fold $(\tau_m+\tau'_m)$-multiplicative.

In particular, the upper bound predicted by Conjecture \ref{conj2.6} is preserved under products.
\end{proposition}
\begin{proof}
Let $p_X : (X\times X')^{m+1}\rightarrow X^{m+1}$ be the natural projection, and let
$p_{X'}: (X\times X')^{m+1}\rightarrow X'^{m+1}$ be the projection onto the remaining factors.
By the definition of the product CK decomposition, one has
\begin{eqnarray*}
   &  & \pi_k^{X\times X'}\circ\Delta_{I}^{X\times X'}\circ (\pi_{i_1}^{X\times X'}\otimes\cdots\otimes\pi_{i_m}^{X\times X'}) \\
   &=& \sum_{\substack{i_{11}+i_{12}=i_1\\ \cdots\cdots\\ i_{m1}+i_{m2}=i_m\\ k_1+k_2=k}} p_X^*\left(\pi_{k_1}^{X}\circ\Delta_{I}^{X}\circ (\pi_{i_{11}}^{X}\otimes\cdots\otimes\pi_{i_{m1}}^{X})\right)\cdot
p_{X'}^*\left(\pi_{k_2}^{X'}\circ\Delta_{I}^{X'}\circ (\pi_{i_{12}}^{X'}\otimes\cdots\otimes\pi_{i_{m2}}^{X'})\right).
\end{eqnarray*}
Suppose first that $k>\sum_{\ell=1}^mi_{\ell}$. For each summand, either $k_1>\sum_{\ell=1}^mi_{\ell 1}$ or $k_2>\sum_{\ell=1}^mi_{\ell 2}$. Similarly, if $k<\sum_{\ell=1}^{m}i_{\ell}-(\tau_m+\tau'_m)$, then, for each summand, either $k_1<\sum_{\ell=1}^mi_{\ell 1}-\tau_m$ or $k_2<\sum_{\ell=1}^mi_{\ell 2}-\tau'_m$.
By the assumed multiplicativity properties of the CK decompositions of $X$ and $X'$, at least one of the two factors in each summand therefore vanishes. It follows that
\begin{equation*}
\pi_k^{X\times X'}\circ\Delta_{I}^{X\times X'}\circ (\pi_{i_1}^{X\times X'}\otimes\cdots\otimes\pi_{i_m}^{X\times X'})=0
\end{equation*}
whenever $k>\sum_{\ell=1}^mi_{\ell}$ or $k<\sum_{\ell=1}^mi_{\ell}-(\tau_m+\tau'_m)$.
Thus the product CK decomposition is $m$-fold $(\tau_m+\tau'_m)$-multiplicative.
\end{proof}

The behavior of the small diagonal embeddings plays a fundamental role in the study of motivic multiple twist-multiplicativity.
The following lemma generalizes Proposition 8.7(iii) of \cite{SV16a}.

\begin{lemma}\label{lem2.18}
Fix an integer $k\geq 2$.
Assume that $X$ admits a $k$-fold $\tau$-multiplicative self-dual CK decomposition $\{\pi_i\}_{0\leq i\leq 2n}$.
Let $\delta_k: X\rightarrow X^k$ denote the $k$-th small diagonal embedding.
Then, for all integers $r$, $s$, $p$ and $t$, one has
\begin{equation}\label{9}
\delta_{k*}\mathrm{CH}_s^r(X)\subseteq\bigoplus_{i=0}^{\tau}\mathrm{CH}_{s+i}^{r+n(k-1)}(X^k),\quad
\delta_k^*\mathrm{CH}_t^p(X^k)\subseteq\bigoplus_{i=0}^{\tau}\mathrm{CH}_{t+i}^p(X).
\end{equation}
\end{lemma}
\begin{proof}
Let $\alpha\in\mathrm{CH}_s^r(X)$. Then
\begin{eqnarray*}
\delta_{k*}\alpha &=& \Delta_{I_k}\circ\pi_{2r-s}^X\circ\alpha \\
   &=& [(\pi_{2n-2r+s}^X\otimes\Delta_{X^k})_*\Delta_{k+1}]\circ\alpha \\
   &=& [(\pi_{2n-2r+s}^X\otimes\Delta_{X^k})_*\Delta_{k+1}]_*\alpha\\
   &=& \sum_{-\tau\leq\sum_{\ell}i_{\ell}-[2n(k-1)+2r-s]\leq 0}[(\pi_{2n-2r+s}^X\otimes\pi_{i_1}^X\otimes\cdots\otimes\pi_{i_k}^X)_*\Delta_{k+1}]_*\alpha\\
   &=& \sum_{-\tau\leq\sum_{\ell}i_{\ell}-[2n(k-1)+2r-s]\leq 0}(\pi_{i_1}^X\otimes\cdots\otimes\pi_{i_k}^X)_*(\delta_{k*}\alpha).
\end{eqnarray*}
The fourth equality follows from the assumed $k$-fold $\tau$-multiplicativity. Therefore, one gets that
$\delta_{k*}\alpha\in\bigoplus_{i=0}^{\tau}\mathrm{CH}_{s+i}^{r+n(k-1)}(X^k)$.

By definition, for any $\beta\in\mathrm{CH}_t^p(X^k)$, one has
\begin{equation*}
\delta_{k}^*\beta\in\bigoplus_{i=0}^{\tau}\text{CH}_{t+i}^{p}(X).
\end{equation*}
\end{proof}
\begin{remark}
\leavevmode
\begin{enumerate}
\item[(i)] Assume instead that $X$ admits a 2-fold $\tau$-multiplicative self-dual CK decomposition.
Then 
\begin{equation}\label{10}
\delta_{k*}\mathrm{CH}_s^r(X)\subseteq\bigoplus_{i=0}^{(k-1)\tau}\mathrm{CH}_{s+i}^{r+n(k-1)}(X^k).
\end{equation}

\item[(ii)] When $\tau\neq 0$, a new phenomenon arises: the decomposition properties of the $(k+1)$-st small diagonal class affect the pushforward induced by the $k$-th small diagonal embedding.
\end{enumerate}
\end{remark}

We now study the behavior of motivic multiple twist-multiplicativity under the formation of projective bundles. 
The following proposition extends Proposition 3.3 of \cite{SV16b}, which corresponds to the case $\tau=0$.
\begin{proposition}\label{prop2.20}
Fix an integer $m\geq 2$.
Let $\mathcal{E}$ be a locally free sheave of rank $e+1\geq 2$ on $X$, and let $\pi: \mathbb{P}(\mathcal{E})\rightarrow X$ be the associated projective bundle.
Suppose that the following conditions hold:
\begin{enumerate}
\item[(i)] The variety $X$ admits a $(m+1)$-fold $\tau$-multiplicative self-dual CK decomposition.

\item[(ii)] For every $i$, the Chern class $c_i(\mathcal{E})\in\mathrm{CH}_0^i(X)$ and intersecting with $c_i(\mathcal{E})$ preserves the grading, that is, for any $\alpha\in\mathrm{CH}_s^r(X)$, then $c_i(\mathcal{E})\cdot\alpha\in\mathrm{CH}_s^{r+i}(X)$.
\end{enumerate}
Then $\mathbb{P}(\mathcal{E})$ admits an $m$-fold $\tau$-multiplicative self-dual CK decomposition
such that
\begin{equation}\label{11}
\mathrm{CH}_s^r(\mathbb{P}(\mathcal{E}))=\bigoplus_{i=0}^e\xi^i\cdot\pi^*\mathrm{CH}_s^{r-i}(X),
\end{equation}
where $\xi\in\mathrm{CH}^1(\mathbb{P}(\mathcal{E}))$ is the first Chern class of the tautological line bundle $\mathcal{O}_{\mathbb{P}(\mathcal{E})}(1)$.
If the Chern classes of $X$ belong to the graded-zero part, then the same is true for the Chern classes of $\mathbb{P}(\mathcal{E})$.
Moreover, the graph correspondence $\Gamma_{\pi}$ is of pure grade 0, that is, $\Gamma_{\pi}\in\mathrm{CH}_0^{n}(\mathbb{P}(\mathcal{E})\times X)$.
\end{proposition}
\begin{proof}
Under assumption (i), the variety $\mathbb{P}(\mathcal{E})$ admits two CK decompositions:
The first is given by formula (6) of \cite{SV16b}, while the second is the self-dual CK decomposition constructed in formula (8) of \cite{SV16b}. The latter construction is compatible with products. By the proof of Proposition 3.3 of \cite{SV16b}, assumption (ii) implies that these two CK decompositions induce the same graded pieces of the Chow groups, namely those in \eqref{11}. It therefore suffices to prove that the self-dual CK decomposition of $\mathbb{P}(\mathcal{E})$ is $m$-fold $\tau$-multiplicative.

By the projective bundle formula for Chow groups, one has
\begin{equation}\label{12}
\Delta_{I}^{\mathbb{P}(\mathcal{E})}
=\sum_{0\leq i_1, \cdots, i_{m+1}\leq e}((\pi^{\times (m+1)})^*(\delta_{m+1}^X)_*\alpha_{i_1\cdots i_{m+1}})\cdot\xi_1^{i_1}\cdots\xi_{m+1}^{i_{m+1}},
\end{equation}
where each $\alpha_{i_1\cdots i_{m+1}}\in\mathrm{CH}^{me-\sum_{\ell=1}^{m+1}i_{\ell}}(X)$ is a polynomial in the Chern classes of $\mathcal{E}$, and $\xi_j=p_j^*\xi\in\text{CH}^1(\mathbb{P}(\mathcal{E})^{m+1})$ is the pull-back of $\xi$ from the $j$-th factor.
By assumption (ii), one has $\alpha_{i_1\cdots i_{m+1}}\in\mathrm{CH}_0^{me-\sum_{\ell=1}^{m+1}i_{\ell}}(X)$.
Assumption (i), together with Lemma \ref{lem2.18}, therefore gives
\begin{equation*}
(\delta_{m+1}^X)_*\alpha_{i_1\cdots i_{m+1}}\in\bigoplus_{s=0}^{\tau}\mathrm{CH}_s^{m(n+e)-\sum_{\ell=1}^{m+1}i_{\ell}}(X^{m+1}).
\end{equation*}
It follows from \eqref{12} that 
\begin{equation*}
\Delta_{I}^{\mathbb{P}(\mathcal{E})}\in\bigoplus_{s=0}^{\tau}\mathrm{CH}_s^{m(n+e)}(\mathbb{P}(\mathcal{E})^{m+1}).
\end{equation*}
Proposition \ref{prop2.15} now implies that the self-dual CK decomposition of $\mathbb{P}(\mathcal{E})$ is $m$-fold $\tau$-multiplicative.

The assertions concerning the Chern classes of $\mathbb{P}(\mathcal{E})$ and the graph correspondence $\Gamma_\pi$ follow from the same argument as in the proof of Proposition 3.3 of \cite{SV16b}.
\end{proof}
\begin{remark}
Without assumption (ii), the conclusion is unlikely to hold in general. Indeed, the behavior of intersection with the Chern classes of $\mathcal{E}$ may have a substantial effect on the motivic multiple multiplicativity defect of $\mathbb{P}(\mathcal{E})$. Even when $X$ admits motivic $0$-multiplicativity, there may exist many vector bundles on $X$ whose associated projective bundles have nonvanishing motivic multiplicativity defect. 
\end{remark}

We now study the behavior of motivic multiple twist-multiplicativity under smooth blow-ups. The following proposition extends Proposition 3.4 of \cite{SV16b}, which corresponds to the case $\tau=0$. 
\begin{proposition}\label{prop2.22}
Fix an integer $m\geq 2$.
Let $Y$ be a smooth connected subvariety of $X$ of codimension $e+1\geq 2$,
and let $\rho: \widetilde{X}\rightarrow X$ be the blow-up of $X$ along $Y$. Consider the following Cartesian square:
\begin{equation}\label{13}
\begin{tikzcd}
E \arrow{r}{\gamma} \arrow{d}{\pi}
&\widetilde{X} \arrow{d}{\rho}\\
Y \arrow{r}{\iota} & X
\end{tikzcd}
\end{equation}
Let $\mathcal{N}:=\mathcal{N}_{\iota}$ denote the normal bundle of $Y$ in $X$.
Assume that the following conditions hold:
\begin{enumerate}
\item[(i)] The variety $X$ (resp. $Y$) admits an $m$-fold (resp. $(m+1)$-fold) $\tau$-multiplicative self-dual CK decomposition.

\item[(ii)] For every $i$, one has $c_i(\mathcal{N})\in\text{CH}_0^i(Y)$ and intersecting with $c_i(\mathcal{N})$ preserves the grading.

\item[(iii)] The correspondence $\Gamma_{\iota}$ is of pure grade 0.
\end{enumerate}

Then $\widetilde{X}$ admits an $m$-fold $\tau$-multiplicative self-dual CK decomposition such that
\begin{equation}\label{14}
\mathrm{CH}_s^r(\widetilde{X})=\rho^*\mathrm{CH}_s^r(X)\oplus\bigoplus_{i=0}^{e-1}
\gamma_*(\xi^i\cdot\pi^*\mathrm{CH}_s^{r-i-1}(Y)),
\end{equation}
where $\xi$ is the first Chern class of the line bundle $\mathcal{O}_{\widetilde{X}}(-E)|_E$.
If the Chern classes of $X$ belong to the graded-zero part, then the same is true for the Chern classes of $\widetilde{X}$. Moreover, the graph correspondences $\Gamma_{\gamma}$ and $\Gamma_{\rho}$ are of pure grade 0.
\end{proposition}
\begin{proof}
Under assumption (i), the blow-up $\widetilde{X}$ admits two CK decompositions given by the formula (11) and the formula (13) of \cite{SV16b}. By the proof of Proposition 3.4 of \cite{SV16b}, together with assumptions (ii) and (iii), these two decompositions induce the same graded pieces of the Chow groups, namely those given in \eqref{14}. The latter CK decomposition is self-dual. It therefore remains to prove that it is $m$-fold $\tau$-multiplicative.

The Cartesian square \eqref{13} induces the following commutative diagram, in which the left square is Cartesian:
\begin{equation}\label{15}
\begin{tikzcd}
E_{/Y}^{m+1} \arrow{r}{\phi} \arrow{d}{\pi_Y^{\times (m+1)}}
& E^{m+1} \arrow{d}{\pi^{\times (m+1)}}\arrow{r}{\gamma^{\times (m+1)}}&\widetilde{X}^{m+1}\\
Y \arrow{r}{\delta_{m+1}^Y} & Y^{m+1}.
\end{tikzcd}
\end{equation}
By assumption (iii) and the construction of the CK decomposition on the blow-up, the graph correspondence $\Gamma_\gamma$ is of pure grade zero. Consequently, so is the external product correspondence
$\Gamma_{\gamma^{\times (m+1)}}$ for each $m\geq 0$.
The argument in the proof of Proposition 3.4 of \cite{SV16b} gives
\begin{eqnarray*}
\mathrm{CH}_s^*(\widetilde{X}^{m+1})
&\supseteq& (\rho^{\times (m+1)})^*\mathrm{CH}_s^*(X^{m+1})\\
&  & +\sum_{i_1=0}^{e-1}\cdots\sum_{i_{m+1}=0}^{e-1}(\gamma^{\times (m+1)})_*(\xi_1^{i_1}\cdots\xi_{m+1}^{i_{m+1}}\cdot (\pi^{\times (m+1)})^*\mathrm{CH}_s^*(Y^{m+1})),
\end{eqnarray*}
where $\xi_j\in\mathrm{CH}^1(E^{m+1})$ is the pull-back of $\xi$ via the projection to the $j$-th factor.
By assumption (ii) and Lemma 3.5 in \cite{SV16b}, one can write
\begin{equation}\label{16}
\Delta_{I}^{\widetilde{X}}=(\rho^{\times (m+1)})^*\Delta_{I}^X+(\gamma^{\times (m+1)})_*\phi_*\alpha,
\end{equation}
where $\alpha=P(\phi^*\xi_j,c((\pi^{\times (m+1)})^*\mathcal{N}))\in\mathrm{CH}^{me-1}(E_{/Y}^{m+1})$ is a polynomial in the classes $\phi^*\xi_j$ and in the pullbacks of the Chern classes of $\mathcal{N}$.
By the projection formula, one gets that
\begin{eqnarray*}
  \phi_*\alpha &=& P(\xi_i, \phi_*(\pi_Y^{\times (m+1)})^*c(\mathcal{N})) \\
   &=& P(\xi_i, (\pi^{\times (m+1)})^*(\delta_{m+1}^Y)_*c(\mathcal{N})) \\
   &=& \sum_{i_1=0}^{e-1}\cdots\sum_{i_{m+1}=0}^{e-1}(\pi^{\times (m+1)})^*((\delta_{m+1}^Y)_*\alpha_{i_1\cdots i_{m+1}})\cdot\xi_1^{i_1}\cdots\xi_{m+1}^{i_{m+1}},
\end{eqnarray*}
where each $\alpha_{i_1\cdots i_{m+1}}\in\mathrm{CH}^{me-1-\sum_{\ell=1}^{m+1}i_{\ell}}(Y)$ is a polynomial of the Chern classes of $\mathcal{N}$.
Assumption (ii) implies that $\alpha_{i_1\cdots i_{m+1}}\in\mathrm{CH}_0^{me-1-\sum_{\ell=1}^{m+1}i_{\ell}}(Y)$.
Since the CK decomposition of $Y$ is $(m+1)$-fold $\tau$-multiplicative, Lemma \ref{lem2.18} gives
\begin{equation*}
(\delta_{m+1}^Y)_*\alpha_{i_1\cdots i_{m+1}}\in\bigoplus_{s=0}^{\tau}\mathrm{CH}_s^{mn-m-1-\sum_{\ell=1}^{m+1}i_{\ell}}(Y^{m+1}).
\end{equation*}
On the other hand, the $m$-fold $\tau$-multiplicativity of the CK decomposition of $X$ implies that
$\Delta_{I}^X\in\bigoplus_{s=0}^{\tau}\mathrm{CH}_s^{mn}(X^{m+1})$.
It follows from \eqref{16} that
\begin{equation*}
\Delta_{I}^{\widetilde{X}}\in\bigoplus_{s=0}^{\tau}\mathrm{CH}_s^{mn}(\widetilde{X}^{m+1}).
\end{equation*}
Proposition \ref{prop2.15} therefore shows that the self-dual CK decomposition of $\widetilde{X}$ is $m$-fold $\tau$-multiplicative.

The assertions concerning the Chern classes of $\widetilde{X}$ and the graph correspondences $\Gamma_{\gamma}$ and $\Gamma_{\rho}$ follow from the same argument as in the proof of Proposition 3.4 of \cite{SV16b}.
\end{proof}

We now begin the study of motivic multiple twist-multiplicativity for concrete classes of varieties and verify Conjecture \ref{conj2.6} for some of the first examples. We first need the following preparatory lemma.

\begin{lemma}\label{lem2.23}
Let $X$ be a smooth connected projective variety satisfying the nilpotence conjecture.
Assume that for a set $\{\pi_i\}_{0\leq i\leq 2n}$ of CK projectors of $X$ and some fixed integer $r\geq 1$,
Murre's conjecture (B) holds in codimension $r$.
Then Murre's conjecture (B) in codimension $r$ holds for any set of CK projectors of $X$.
\end{lemma}
\begin{proof}
Let $\{\widetilde{\pi}_i\}_{0\leq i\leq 2n}$ be any set of CK projectors of $X$.
The nilpotence conjecture for $X$ implies that the ideal $\mathrm{CH}^n(X\times X)_{\text{hom}}$ of the correspondence ring $\mathrm{CH}^n(X\times X)$ is nilpotent. By Lemma 5.4 in \cite{Ja94}, there exists a correspondence $\eta\in\mathrm{CH}^n(X\times X)_{\text{hom}}$ such that
$\pi_i=(1+\eta)\circ\pi'_i\circ(1+\eta)^{-1}$. Then $\mathfrak{h}^i(X):=(X, \pi_i)\cong (X,\pi'_i)=:\widetilde{\mathfrak{h}}^i(X)$.
By assumption, one has
\begin{equation*}
0=\pi_{i*}\mathrm{CH}^r(X)=\text{Hom}(\mathbb{L}^r, \mathfrak{h}^i(X))\cong\text{Hom}(\mathbb{L}^r, \widetilde{\mathfrak{h}}^i(X))=\widetilde{\pi}_{i*}\mathrm{CH}^r(X)
\end{equation*}
for every integer $i<r$ or $i>2r$. Thus Murre's conjecture (B) in codimension $r$ also holds for the set $\{\widetilde{\pi}_i\}_{0\leq i\leq 2n}$ of CK projectors.
\end{proof}

As a first test case, we consider curves.
\begin{theorem}\label{thm2.24}
Conjecture \ref{conj2.6} holds for any smooth connected projective curve.
Moreover, let $X$ be a curve of genus $g\geq 2$.
\begin{enumerate}
\item[(i)] If $2\leq m\leq g$, then $\mathfrak{d}_m(X)\leq m-1$.

\item[(ii)] For every integer $m\geq g$, one has $\mathfrak{d}_m(X)\leq g-1$ and hence $\mathfrak{sd}(X)\leq g-1$.
\end{enumerate}
\end{theorem}
\begin{proof}
Since the Chow motive of $X^{m+1}$ is finite-dimensional, then the nilpotence conjecture holds for $X^{m+1}$. 
By Lemma \ref{lem2.23}, it suffices to verify Murre's conjecture (B) in codimension $m$ for one particular set of CK projectors of $X^{m+1}$.

Choose a closed point $a\in X$, and let $o:=[a]\in\mathrm{CH}^1(X)$. 
Set
\begin{equation*}
\pi_0=o\times X,\quad \pi_2=X\times o,\quad \pi_1=\Delta_{X}-\pi_0-\pi_2\in\mathrm{CH}^1(X\times X).
\end{equation*}
Then $\{\pi_i\}_{0\leq i\leq 2}$ forms a set of self-dual CK projectors of $X$.
Equip $X^{m+1}$ with the product CK projectors 
$\{\Pi_k=\sum_{\sum_{\ell=0}^{m}i_{\ell}=k}\pi_{i_0}\otimes\cdots\otimes\pi_{i_{m}}\}_{0\leq k\leq 2(m+1)}$.
We will verify Murre's conjecture (B) in codimension $mn$ for these projectors, i.e.,
\begin{equation}\label{17}
\Pi_{k*}\mathrm{CH}^{m}(X^{m+1})
=\text{Hom}(\mathbb{L}^{m}, \mathfrak{h}^{k}(X^{m+1})=0
\end{equation}
for every $k<m$ or $k>2m$. This would prove Conjecture \ref{conj2.6} for $X$.
\medskip

\emph{Case (a): $k>2m$.}
Since $k-(m+1)>m-1$, then $i_{\ell}=2$ for some $\ell$.
By symmetry, we may assume $i_{0}=2$.
Thus,
\begin{equation*}
\text{Hom}(\mathbb{L}^{m}, \mathfrak{h}^{2}(X)\otimes\mathfrak{h}^{i_{1}}(X)\otimes\cdots\otimes\mathfrak{h}^{i_{m}}(X))
=\text{Hom}(\mathbb{L}^{m-1}, \mathfrak{h}^{i_{1}}(X)\otimes\cdots\otimes\mathfrak{h}^{i_{m}}(X)),
\end{equation*}
because $\mathfrak{h}^{2}(X)\cong\mathbb{L}$.
Note that $\sum_{\ell1}^{m}i_{\ell}=k-2>2(m-1)$.
By induction on $m$, one is reduced to show that
\begin{equation*}
\text{Hom}(\mathbb{L}, \mathfrak{h}^{i_{m-1}}(X)\otimes\mathfrak{h}^{i_{m}}(X))=0
\end{equation*}
whenever $i_{m-1}+i_m>2$, which is immediate.
\medskip

\emph{Case (b): $k<m$.}
If $i_{\ell}=2$ for some $\ell$, then we may assume $i_0=2$ by symmetry.
Thus,
\begin{equation*}
\text{Hom}(\mathbb{L}^{m}, \mathfrak{h}^{2}(X)\otimes\mathfrak{h}^{i_1}(X)\otimes\cdots\otimes\mathfrak{h}^{i_{m}}(X))
\cong\text{Hom}(\mathbb{L}^{m-1}, \mathfrak{h}^{i_1}(X)\otimes\cdots\otimes\mathfrak{h}^{i_{m}}(X)).
\end{equation*}
By induction, one is reduced to the case that each $i_{\ell}\leq 1$.
Since $k<m$, then $i_{\ell}=0$ for some $\ell$.
We may assume $i_{0}=0$.
It is enough to show that
\begin{equation*}
\text{Hom}(\mathbb{L}^{m}, \mathfrak{h}^{i_0}(X)\otimes\cdots\otimes\mathfrak{h}^{i_{m}}(X))
\cong\text{Hom}(\mathbb{L}^{m}, \mathfrak{h}^{i_1}(X)\otimes\cdots\otimes\mathfrak{h}^{i_{m}}(X))=0.
\end{equation*}
Note that the group $\mathrm{CH}^{m}(X^{m})$ is generated by 0-cycle classes of the form $z_1\times\cdots\times z_{m}$ with $z_i\in\mathrm{CH}^1(X)$. Since $\sum_{\ell=1}^{m}i_{\ell}=k<m$, there exists an index $1\leq\ell\leq m$ such that $i_\ell=0$. The vanishing in \eqref{17} follows from $\pi_{0*}z_{\ell}=0$.

We now assume that $X$ has genus $g\geq 2$ and prove assertions (i) and (ii).

\medskip
\emph{The case $m=2$.} 
Denote $\Delta_{123}:=\Delta_I$, $\Delta_{ij}:=p_{ij}^*\Delta_X\cdot p_k^*o_X$ 
and $\Delta_i:=p_i^*[X]\cdot p_j^*o_X\cdot p_k^*o_X$ for $\{i,j,k\}=\{1,2,3\}$.
Straightforward computations give that
\begin{equation}\label{18}
\pi_0\circ\Delta_{123}\circ (\pi_0\otimes\pi_0)
=(\pi_0\otimes\pi_2\otimes\pi_2)_*\Delta_{123}
=\Delta_1,
\end{equation}
\begin{equation}\label{19}
\Delta_{123}\circ (\pi_2\otimes\pi_0)
=(\Delta_X\otimes\pi_0\otimes\pi_2)_*\Delta_{123}=\Delta_2,
\end{equation}
\begin{equation}\label{20}
\Delta_{123}\circ (\pi_0\otimes\pi_2)
=(\Delta_X\otimes\pi_2\otimes\pi_0)_*\Delta_{123}=\Delta_3.
\end{equation}
Then it follows that for $k\neq i+j$ and $(i,j,k)\neq (1,1,1)$, one has
\begin{equation}\label{21}
\pi_k\circ\Delta_{123}\circ(\pi_i\otimes\pi_j)=0.
\end{equation}
Denote the third modified diagonal class of $(X, o_X)$ by
\begin{equation*}
\Gamma^3(X,o_X):=\Delta_{123}-\Delta_{12}-\Delta_{13}-\Delta_{23}+\Delta_1-\Delta_2+\Delta_3\in\mathrm{CH}^2(X^3).
\end{equation*}
Then by \eqref{21}, the above CK decomposition is 1-multiplicative.
Moreover, by the equalities \eqref{18}-\eqref{21}, one has
\begin{equation*}
\Gamma^3(X,o_X)=\pi_1\circ\Delta_{123}\circ(\pi_1\otimes\pi_1)
=(\pi_1\otimes\pi_1\otimes\pi_1)_*\Delta_{123}\in\mathrm{CH}_1^2(X^3).
\end{equation*}
\medskip

\emph{The case $m\geq 3$.}
By definition, $\mathfrak{d}_m(X)\leq (m-1)\mathfrak{d}(X)\leq m-1$.
We give a more explicit description of the relevant decomposition.
We will show that if $k\neq\sum_{\ell=1}^{m}i_{\ell}$ with $k=0$ or some $i_{\ell'}=2$,
then one has
\begin{equation*}
\pi_k\circ\Delta_{I}\circ(\pi_{i_1}\otimes\cdots\otimes\pi_{i_m})=0.
\end{equation*}
Indeed, if some $i_{\ell'}=2$, we may assume $i_{1}=2$ by symmetry.
Thus
\begin{eqnarray*}
 \pi_k\circ\Delta_{I}\circ(\pi_{i_1}\otimes\cdots\otimes\pi_{i_m}) &=& (\pi_{0}\otimes\pi_{2-i_2}\otimes\cdots\otimes\pi_{2-i_m}\otimes\pi_k)_*\Delta_I \\
    &=& X\times(\pi_{2-i_2})_*o_X\times\cdots(\pi_{2-i_m})_*o_X\times\pi_{k*}o_X\\
    &=& 0.
\end{eqnarray*}
The last equality follows from $\pi_{0*}o_X=0=\pi_{1*}o_X$ and from the fact that if $i_{\ell}=0$ for all $\ell\geq 2$, 
then the condition $k\neq\sum_{\ell=1}^{m}i_\ell$ implies that $k\neq 2$. The case $k=0$ is similar.

Hence, the above CK decomposition is $m$-fold $(m-1)$-multiplicative.
Moreover, one gets that 
\begin{equation}\label{22}
\Delta_I-\sum_{k=0}^{4}\sum_{\sum_{\ell=1}^{m}i_{\ell}=k}\pi_k\circ\Delta_I\circ(\pi_{i_1}\otimes\cdots\otimes\pi_{i_m})
=\sum_{s=0}^{m-1}\Gamma_s,
\end{equation}
where 
\begin{equation}\label{23}
\Gamma_s=\overbrace{(\pi_1\otimes\cdots\otimes\pi_1}^{s+2})_*\Delta_{I_{s+1}}\times 
\overbrace{o_X\times\cdots\times o_X}^{m-s-1}+(\text{permutations})\in\mathrm{CH}_{s}^{m}(X^{m+1}).
\end{equation}
Hence,
\begin{equation}\label{24}
\Gamma^{m+1}(X,o_X)
=(\overbrace{\pi_1\otimes\cdots\otimes\pi_{1}}^{m+1})_*\Delta_{I}\in\mathrm{CH}_{m-1}^m(X^{m+1}).
\end{equation}

It follows from \cite{MY16} that $\Gamma^{g+2}(X,o_X)=0$. 
On the other hand, the preceding argument gives
\begin{equation*}
\Delta_{I_g}\in\bigoplus_{s=0}^{g-1}\mathrm{CH}_{s}^g(X^{g+1}).
\end{equation*}
The proof of Proposition \ref{prop2.10} therefore shows that the CK decomposition of $X$ is $m$-fold $(g-1)$-multiplicative for every integer $m\geq g$. This proves both assertions.
\end{proof}

Using the nonvanishing results for modified diagonal classes established in \cite{Y13}, we can determine the motivic multiple multiplicativity defects of a very general curve.
\begin{corollary}
Let $X$ be a very general curve of genus $g\geq 3$. 
\begin{enumerate}
\item[(i)] For every integer $m$ with $2\leq m\leq g-1$, one has $\mathfrak{d}_m(X)=m-1$. In particular, $\mathfrak{d}(X)=1$.

\item[(ii)] For every integer $m\geq g$, one has $\mathfrak{d}_m(X)=g-1$.
In particular, $\mathfrak{sd}(X)=g-1$.
\end{enumerate}
\end{corollary}
\begin{proof}
The proof of Theorem \ref{thm2.24} shows that the required lower bounds are detected by the nonvanishing of the corresponding modified diagonal classes. According to \cite{BV04,Y13}, if $X$ is a very general curve of genus $g\geq 3$, then $\Gamma^{m+1}(X,o_X)\neq 0$ for every $m\in\{2, 3\cdots, g\}$ and any degree-one 0-cycle class $o_X$.
Therefore, $\mathfrak{d}_m(X)=m-1$ for every integer $m\in\{2, 3\cdots, g-1\}$,
while $\mathfrak{d}_m(X)=g-1$ for every integer $m\geq g$.
\end{proof}

\begin{remark}
\leavevmode
\begin{enumerate}
\item[(i)] For certain special non-hyperelliptic curve of genus at least $3$, the motivic multiplicativity defect may drop to zero under specialization. It would be interesting to classify all such curves.

\item[(ii)] Modulo algebraic equivalence, it is expected that the stable motivic multiplicativity defect $\mathfrak{sd}_a(X)=\mathrm{gon}(X)-2=\lfloor\frac{g-1}{2}\rfloor$, where $\mathrm{gon}(X)=\lfloor\frac{g+3}{2}\rfloor$ is the gonality of $X$. We refer to Section 4.4 of \cite{MY16} for further details. As supporting evidence, Beauville \cite{Bea23} constructed a non-hyperelliptic curve of genus $3$ whose motivic multiplicativity defect vanishes modulo algebraic equivalence. There should exist other non-hyperelliptic curves also expected to have this property; see \cite{BLLS23}.

\item[(iii)] Another natural approach to detecting the nonvanishing of modified diagonal classes would be to use the higher Abel--Jacobi maps for Chow groups introduced in \cite{As00,S01}.

\item[(iv)] Let $X_i(1\leq i\leq k)$ be very general curves of genus at least $3$.
Then it is not hard to see that $\mathfrak{d}(\prod_{i=1}^{k}C_i)=k$. 
\end{enumerate}
\end{remark}

We next turn to the case of surfaces.

\begin{theorem}\label{thm2.27}
Let $X$ be a smooth connected projective surface. 
Then for every integer $m\geq 2$, Murre's conjecture (B) holds in codimension $2m$ for some set of CK projectors of $X^{m+1}$. As a consequence, $\mathfrak{d}_m(X)\leq 2m$.
If, moreover, $X$ is regular, then $\mathfrak{d}_m(X)$ is even and $\mathfrak{d}_m(X)\leq 2m-2$.
Furthermore, if $X$ is swept-out by irreducible curves of genus $g$, each supporting a 0-cycle rationally equivalent to a fixed degree-one 0-cycle class $o_X$, then $\mathfrak{sd}(X)\leq 6g+2$.
\end{theorem}
\begin{proof}
By \cite{Mu90}\cite{Sch94}, the surface $X$ admits a self-dual CK decomposition:
$\Delta_X=\sum_{i=0}^{4}\pi_i$, where $\pi_0=o_X\times X$ for some degree-one 0-cycle class $o_X$.
For every $m\geq 2$, equip $X^{m+1}$ with the corresponding product CK projectors.
We shall verify Murre's conjecture (B) for these CK projectors in codimension $2m$, that is,
\begin{equation}\label{25}
(\pi_{k_0}\otimes\cdots\otimes\pi_{k_m})_*\text{CH}^{2m}(X^{m+1})
=\text{Hom}(\mathbb{L}^{2m}, \mathfrak{h}^{k_0}(X)\otimes\cdots\otimes\mathfrak{h}^{k_{m}}(X))=0
\end{equation}
whenever $\sum_{\ell=0}^{m}k_{\ell}>4m$ or $\sum_{\ell=0}^{m}k_{\ell}<2m$.
\medskip

\emph{Case (i): $\sum_{\ell=0}^{m}k_\ell>4m$.}
Suppose first that $k_\ell=4$ for some $\ell$. By symmetry, we may assume that $k_0=4$. Since $\mathfrak{h}^{4}(X)\cong\mathfrak{h}^0(X)\otimes\mathbb{L}^2\cong\mathbb{L}^2$, we have
\begin{equation*}
\text{Hom}(\mathbb{L}^{2m}, \mathfrak{h}^{4}(X)\otimes\mathfrak{h}^{k_1}(X)\otimes\cdots\otimes\mathfrak{h}^{k_{m}}(X))
=\text{Hom}(\mathbb{L}^{2m-2}, \mathfrak{h}^{k_1}(X)\otimes\cdots\otimes\mathfrak{h}^{k_{m}}(X)).
\end{equation*}
By induction, one is reduced to the case $m=1$ which follows by results in \cite{Mu93}, or to the case that each $k_{\ell}\leq 3$ for every $\ell$. In the latter case, $4m+1\leq\sum_{\ell=0}^{m}k_\ell\leq3m+3$.
This forces $m=2$ and $k_0=k_1=k_2=3$.
By Theorem 4.4 (ii) in \cite{Sch94}, $\mathfrak{h}^3(X)\cong\mathfrak{h}^1(X)\otimes\mathbb{L}^1$. 
It follows that
\begin{eqnarray*}
   \text{Hom}(\mathbb{L}^4, \mathfrak{h}^3(X)\otimes\mathfrak{h}^3(X)\otimes\mathfrak{h}^3(X))
   &\cong& \text{Hom}(\mathbb{L}^4, \mathfrak{h}^1(X)\otimes\mathfrak{h}^1(X)\otimes\mathfrak{h}^1(X)
   \otimes\mathbb{L}^3) \\
   &=& \text{Hom}(\mathbb{L}, \mathfrak{h}^1(X)\otimes\mathfrak{h}^1(X)\otimes\mathfrak{h}^1(X))\\
   &=& (\pi_1\otimes\pi_1\otimes\pi_1)_*\text{CH}^1(X^3)\\
   &=& 0.
\end{eqnarray*}
\medskip

\emph{Case (ii): $\sum_{\ell=0}^{m}k_\ell<2m$.}
At least one of the $k_\ell$ is at most $1$, since otherwise
$\sum_{\ell=0}^{m}k_\ell\geq 2(m+1)>2m$.
By symmetry, we may assume that $k_0\leq 1$.
One has $k_{\ell}\leq 1$ for some $\ell$, because $2m+2>2m-1$.
Then we may assume $k_0\leq 1$ by symmetry.
If $k_0=0$, then
$$
\text{Hom}(\mathbb{L}^{2m}, \mathfrak{h}^{0}(X)\otimes\mathfrak{h}^{k_1}(X)\otimes\cdots\otimes\mathfrak{h}^{k_{m}}(X))
\cong\text{Hom}(\mathbb{L}^{2m}, \mathfrak{h}^{k_1}(X)\otimes\cdots\otimes\mathfrak{h}^{k_{m}}(X)).
$$
Now since $\sum_{\ell=1}^{m}k_{\ell}<2m$, we may repeat this reduction and assume that every $k_\ell\geq 1$. We may therefore assume that $k_0=1$.
It remains to prove that
\begin{eqnarray*}
   & &\text{Hom}(\mathbb{L}^{2m}, \mathfrak{h}^{1}(X)\otimes\mathfrak{h}^{k_1}(X)\otimes\cdots\otimes\mathfrak{h}^{k_{m}}(X))\\
   &\cong& \text{Hom}(\mathbb{L}^{2m}, \mathfrak{h}^{1}(A)\otimes\mathfrak{h}^{k_1}(X)\otimes\cdots\otimes\mathfrak{h}^{k_{m}}(X)) \\
   &=& 0,
\end{eqnarray*}
where $A$ is the Picard variety of $X$ and $\mathfrak{h}^1(X)\cong\mathfrak{h}^1(A)$.
It is well known that there exists a smooth connected projective curve $C$, an integer $n\geq 1$, and a surjective morphism $C^n\rightarrow A$. Consequently, $\mathfrak{h}^1(A)$ is a direct summand of $\mathfrak{h}^1(C^n)$.
It therefore suffices to show that 
\begin{equation}\label{26}
\text{Hom}(\mathbb{L}^{2m}, \mathfrak{h}^{1}(C^n)\otimes\mathfrak{h}^{k_1}(X)\otimes\cdots\otimes\mathfrak{h}^{k_{m}}(X))=0.
\end{equation}
This is readily reduced to showing that
\begin{equation}\label{27}
\text{Hom}(\mathbb{L}^{2m}, \mathfrak{h}^{1}(C)\otimes\mathfrak{h}^{k_1}(X)\otimes\cdots\otimes\mathfrak{h}^{k_{m}}(X))=0.
\end{equation}
By the argument of Theorem 2.4 in \cite{XX13}, it is enough to establish the vanishings:
\begin{equation}\label{28}
\text{Hom}(\mathbb{L}^{2m}, \mathfrak{h}^{k_1}(X)\otimes\cdots\otimes\mathfrak{h}^{k_{m}}(X))=0
\end{equation}
and
\begin{equation}\label{29}
\text{Hom}(\mathbb{L}^{2m-1}, \mathfrak{h}^{k_1}(X)\otimes\cdots\otimes\mathfrak{h}^{k_{m}}(X))=0.
\end{equation}
Repeating this procedure, we are eventually reduced to proving that
\begin{equation}\label{30}
\text{Hom}(\mathbb{L}^{j}, \mathfrak{h}^{1}(X)\otimes\mathfrak{h}^{2}(X)^{\otimes r})=0
\end{equation}
for some integer $r\geq 0$ and every integer $j\geq 2r+2$.

For $j>2r+2$, the vanishing in \eqref{30} follows for dimensional reasons.
For $j=2r+2$, then the group $\mathrm{CH}^{2r+2}(X^{r+1})$ is generated by 0-cycle classes of the form $z_1\times\cdots\times z_{r+1}$ with $z_i\in\mathrm{CH}^2(X)$.
Then the vanishing in \eqref{30} follows from the fact that 
\begin{equation*}
  (\pi_1\otimes\overbrace{\pi_2\otimes\cdots\otimes\pi_2}^r)_*(z_1\times\cdots\times z_{r+1})
=\pi_{1*}z_1\times\pi_{2*}z_2\cdots\times\pi_{2*}z_{r+1}=0.
\end{equation*}
We have therefore established Murre's conjecture (B) in codimension $2m$ for the chosen product CK decomposition of $X^{m+1}$. Consequently, $\mathfrak{d}_m(X)\leq 2m$. 
If $X$ is swept out by irreducible curves of genus $g$, each supporting a $0$-cycle rationally equivalent to $o_X$, then Remark \ref{rem2.13}(i) gives $\mathfrak{sd}(X)\leq 6g+2$.

Now suppose that $X$ is regular.
Then $\pi_1=0=\pi_3$ and hence $\mathfrak{d}_m(X)$ is even.
The same argument just as in the case of curves implies that 
\begin{equation}\label{31}
\Gamma^{m+1}(X,o_X)
=(\overbrace{\pi_2\otimes\cdots\otimes\pi_2}^{m+1})_*\Delta_{I}\in\mathrm{CH}_{2(m-1)}^{2m}(X^{m+1}).
\end{equation}
Then 
\begin{equation}\label{32}
\Delta_I-\sum_{k=0}^{4}\sum_{\sum_{\ell=1}^{m}i_{\ell}=k}\pi_k\circ\Delta_I\circ(\pi_{i_1}\otimes\cdots\otimes\pi_{i_m})
=\sum_{s=0}^{m-1}\Gamma_s,
\end{equation}
where 
\begin{equation*}
\Gamma_s=\overbrace{(\pi_2\otimes\cdots\otimes\pi_2}^{s+2})_*\Delta_{I_{s+1}}\times 
\overbrace{o_X\times\cdots\times o_X}^{m-s-1}+(\text{permutations})\in\mathrm{CH}_{2s}^{2m}(X^{m+1}).
\end{equation*}
It follows that $\mathfrak{d}_m(X)\leq 2(m-1)$.
\end{proof}
\begin{remark}
\leavevmode
\begin{enumerate}
\item[(i)] Unlike the case of curves, it is unclear how to determine a sharp upper bound for the stable motivic multiplicativity defect of a very general surface.

\item[(ii)] For a regular surface, the motivic multiplicativity defect is either $0$ or $2$. Thus, if such a surface is not strictly motivic $2$-multiplicative, then it admits motivic $0$-multiplicativity.

\item[(iii)] Unfortunately, the preceding argument does not extend to arbitrary varieties of dimension at least $3$, since very little is known about algebraic cycles on their powers.
\end{enumerate}
\end{remark}

We can now deduce the following existence result.
\begin{proposition}
Fix an integer $n\geq 1$. For every integer $\tau$ satisfying $0\leq\tau\leq n$,
there exists a smooth connected projective variety of dimension $n$ whose motivic multiplicativity defect is equal to $\tau$.
\end{proposition}
\begin{proof}
For each nonnegative integer $i\leq\tau$, let $C_i$ be a very general curve of genus at least $3$, and
set $X:=\prod_{i=1}^{\tau}C_i\times\mathbb{P}^{n-\tau}$.
When $\tau=0$, the product of curves is understood to be a point, so that $X=\mathbb{P}^n$.
Applying Proposition \ref{prop2.20} to the trivial projective bundle shows that
$\mathfrak{d}(X)=\tau$.
\end{proof}

Stable motivic multiplicativity immediately yields vanishing results for modified diagonal classes. From this perspective, stable motivic multiplicativity may be regarded as a natural generalization of motivic $0$-multiplicativity. The following result generalizes Proposition 8.12 of \cite{SV16a}.
\begin{proposition}\label{prop2.30}
Assume that $X$ admits a set of self-dual CK projectors $\{\pi_i^X\}_{0\leq i\leq 2n}$ with $\pi_{0}^X=o_X\times X$,
and suppose that this CK decomposition attains the stable motivic multiplicativity defect $\mathfrak{sd}(X)=\tau$.
Then $\Gamma^{k}(X, o_X)=0$  for every integer $k>2n+\tau$.
If, in addition, $\pi_1^X=0$, then $\Gamma^{k}(X, o_X)=0$ for every integer $k>n+\frac{\tau}{2}$.
\end{proposition}
\begin{proof}
By assumption, for every $k\geq 2$, the CK decomposition is $k$-fold $\tau$-multiplicative.
So, the small diagonal class $\Delta_{I_{k-1}}\in\bigoplus_{s=0}^{\tau}\mathrm{CH}_s^{n(k-1)}(X^k)$.
Then
\begin{equation*}
(\pi_{i_1}\otimes\cdots\otimes\pi_{i_k})_*\Delta_k=0
\end{equation*}
whenever $\sum_{\ell=1}^{k}i_{\ell}<2n(k-1)-\tau$ or $\sum_{\ell=1}^{k}i_{\ell}>2n(k-1)$.

Suppose that $k>2n+\tau$.
If
\begin{equation*}
2n(k-1)-\tau\leq\sum_{\ell=1}^{k}i_{\ell}\leq 2n(k-1),
\end{equation*}
then at least one of the indices $i_\ell$ must be equal to $2n$. Indeed, otherwise 
\begin{equation*}
\sum_{\ell=1}^{k}i_{\ell}\leq k(2n-1)<2n(k-1)-\tau,
\end{equation*}
which is a contradiction. It follows from the standard projector expression for the modified diagonal that $\Gamma^k(X, o_X)=0$ for every integer $k>2n+\tau$.

Now suppose that $\pi_1^X=0$ and $k>n+\frac{\tau}{2}$.
By self-duality, $\pi_{2n-1}^X=0$.
If none of the indices $i_\ell$ is equal to $2n$, then every nonzero projector occurring has index at most $2n-2$, and hence
\begin{equation*}
2n(k-1)-\tau\leq\sum_{\ell=1}^{k}i_{\ell}\leq k(2n-2)<2n(k-1)-\tau.
\end{equation*}
This again contradicts the required range. Therefore, some $i_\ell$ must equal $2n$, and consequently
$\Gamma^k(X, o_X)=0$ for every integer $k>n+\frac{\tau}{2}$.
\end{proof}

\begin{remark}
\leavevmode
\begin{enumerate}
\item[(i)] Voisin \cite{Voi15} proved that, if $X$ is swept out by irreducible curves of genus $g$, each supporting a $0$-cycle rationally equivalent to $o_X$, then $\Gamma^m(X, o_X)=0$ for every integer $m\geq (n+1)(g+1)$.
Except in certain special cases, such as rationally connected varieties, the curves in question are generally taken to be complete intersection curves on $X$, whose genera may be very large. Thus, Voisin's result does not generally provide a sharp vanishing threshold for modified diagonal classes. Determining the stable motivic multiplicativity defect appears to be necessary for approaching the optimal threshold.

\item[(ii)] It is reasonable to expect that Proposition \ref{prop2.30} is sharp for many classes of varieties satisfying its hypotheses. It would be interesting to determine the optimal vanishing threshold for modified diagonal classes of an arbitrary variety.

\item[(iii)] Even if a variety satisfies an optimal modified diagonal vanishing property, its stable motivic multiplicativity defect may still be large, as illustrated by rationally connected varieties.
\end{enumerate}
\end{remark}

\begin{remark}
By contrast, the bound obtained by Voisin's method in \cite{Voi15} is
\begin{equation}\label{33}
n+(n+1)\left(\sum_{m=1}^{r}(-1)^{m+1}\sum_{1\leq i_1<\cdots<i_m\leq r}\binom{r+1-\sum_{k=1}^md_{i_k}}{r+1}\right).
\end{equation}
\end{remark}

For regular varieties, motivic $1$-multiplicativity also yields the vanishing of certain modified diagonal classes.
The following result is a slight generalization of Proposition 8.12 of \cite{SV16a}.

\begin{proposition}\label{prop2.33}
Assume that $X$ admits a $1$-multiplicative self-dual CK decomposition satisfying $b_1(X)=0$ and $\pi_{2n}=X\times o_X$. Then $\Gamma^k(X, o_X)=0$ for every integer $k>2n-2$.
\end{proposition}
\begin{proof}
By assumption, one has $\mathfrak{d}_{k-1}(X)\leq k-2$ for every integer $k\geq 3$.
By Proposition \ref{prop2.15}, one has 
\begin{equation*}
\Delta_{I_{k-1}}=\delta_{k*}[X]\in\bigoplus_{s=0}^{k-2}\mathrm{CH}_{s}^{n(k-1)}(X^k).
\end{equation*}
Then
\begin{equation*}
(\pi_{i_1}\otimes\cdots\otimes\pi_{i_k})_*\Delta_{I_{k-1}}=0
\end{equation*}
whenever $\sum_{\ell=1}^{k}i_{\ell}<2n(k-1)-(k-2)$ or $\sum_{\ell=1}^{k}i_{\ell}>2n(k-1)$.

Since $b_1(X)=0$ and the CK decomposition is self-dual, one has $\pi_1=0=\pi_{2n-1}$.
Now suppose that $k>2n-2$.
If 
\begin{equation*}
2n(k-1)-(k-2)\leq\sum_{\ell=1}^{k}i_{\ell}\leq 2n(k-1),
\end{equation*}
then at least one index $i_\ell$ must equal $2n$. Indeed, otherwise every nonzero index is at most $2n-2$, and hence
$\sum_{\ell=1}^{k}i_\ell\leq k(2n-2)<2n(k-1)-(k-2)$, a contradiction.
It follows that $\Gamma^k(X, o_X)=0$ for every integer $k>2n-2$.
\end{proof}
\bigskip

\section{Isogenous correspondences and cycle relations}
\bigskip

In this preparatory section, we construct isogenous correspondences for Fano and Calabi--Yau complete intersections in certain ambient varieties. We then establish the associated cycle relations, which will play an important role in determining the motivic multiple multiplicativity defects of these varieties.

The notion of a $K$-correspondence was introduced by Voisin \cite{Voi04} in connection with a problem of Kobayashi, and $K$-correspondences have been constructed for certain typical Calabi--Yau varieties in \cite{Voi04}. In the Calabi--Yau case considered below, the basic idea underlying our construction is essentially due to Voisin. For our purposes, however, we require a slightly more flexible notion.
\begin{definition}
Let $X$ be a smooth connected projective variety of dimension $n$.
A correspondence $\Sigma\in\text{CH}^n(X\times X)$ is said to be \emph{isogenous}, if it admits a representative $\sum_i m_i\Sigma_i$ such that every irreducible component $\Sigma_i$ of its support dominates both factors under the natural projections.

An isogenous correspondence is called a \emph{$K$-autocorrespondence} if, in addition, for every $i$ and every desingularization $\tau:\widetilde{\Sigma}_i\rightarrow\Sigma_i$, the two composite morphisms $f_k=p_k\circ\tau: \widetilde{\Sigma}_i\rightarrow X$ ($k=1, 2$) have the same ramification divisor.
\end{definition}

We now fix the geometric setting. Let $Y$ be a smooth connected projective variety of dimension $n+1\geq 3$. We shall mainly consider the following two cases:
\medskip

\noindent
\emph{Case (a).}
The variety $Y$ is rationally connected, and $X\subseteq Y$ is a smooth anti-canonical divisor. Then $X$ is a Calabi-Yau variety of dimension $n$.

\noindent
\emph{Case (b).}
The variety $Y$ is Fano of index $i_Y\geq 2$. Thus, there exists a fundamental ample divisor $H$ on $Y$ such that
$-K_Y=i_Y\cdot H$ in $\mathrm{CH}^1(Y)$.
Now let $X\in |kH|$ be a smooth member for some positive integer $k<i_Y$.
By the adjunction formula, $X$ is a Fano variety of index $i_X=i_Y-k\geq 1$.

Let $\iota: X\rightarrow Y$ denote the inclusion. In Case (a), the pullback 
$\iota^*: H^n(Y, \mathbb{Q})\rightarrow H^n(X, \mathbb{Q})$ of singular cohomology groups is not an isomorphism. 
For our purposes, we shall assume that the same condition holds in Case (b).

Choose a smooth rational curve $R\subseteq Y$.
In Case (a), set $d:=-\int_Y K_Y\cdot R\geq 2$.
In Case (b), set $d_0:=\int_Y H\cdot R\geq 2$ and $d:=\int_YX\cdot R=kd_0$.
Fix an ordered pair of positive integers $\alpha:=(m_1,m_2)$ such that $m_1+m_2=d$.

For a smooth variety $W$, let $W^{[d]}$ denote the Hilbert scheme of length-$d$ closed subschemes of $W$, 
and let $W_0^{[d]}$ denote its principal component.
We write $\pi_0^W=[\cdot]: W_0^{[d]}\rightarrow W^{(d)}$ for the Hilbert-Chow morphism.
Let $\text{Hilb}_{R}(Y)$ denote the irreducible component of the Hilbert scheme of $Y$ containing the point $[R]$ and parameterizing deformations of $R$. Its general member is a smooth rational curve.

We now construct several related varieties by imposing prescribed intersection conditions. Define the following reduced closed subvarieties:
\begin{equation*}
\Theta_{\alpha}:=\{(x,y, C)\in X\times X\times\text{Hilb}_{R}(Y)|[X\cap C]=m_1x+m_2y,\ \text{or}\ C\subseteq X\},
\end{equation*}
\begin{equation*}
\Lambda_{\alpha}:=\{C\in\text{Hilb}_{R}(Y)|[X\cap C]=m_1x+m_2y,\ \text{or}\ C\subseteq X\},
\end{equation*}
\begin{equation*}
\Sigma_{\alpha}:=\{(x,y)\in X\times X|[X\cap C]=m_1x+m_2y,\ \text{or}\ C\subseteq X\}.
\end{equation*}
Let $q_1: \Theta_{\alpha}\rightarrow X$, $q_2: \Theta_{\alpha}\rightarrow X$, $q_{12}: \Theta_{\alpha}\rightarrow X\times X$ and $q_3: \Theta_{\alpha}\rightarrow\text{Hilb}_{R}(Y)$ be the natural projections.
By construction,
\begin{equation*}
\Lambda_{\alpha}=q_3(\Theta_{\alpha}), \ \Sigma_{\alpha}=q_{12}(\Theta_{\alpha}).
\end{equation*}
Moreover, there is a dominant rational map
\begin{equation*}
\varphi_{\alpha}: \Lambda_{\alpha}\dashrightarrow\Sigma_{\alpha}
\end{equation*}
defined generically by $\varphi_{\alpha}(C)=(x,y)$ whenever $[X\cap C]=m_1x+m_2y$. 
Consequently,
\begin{equation*}
\text{dim}\Theta_{\alpha}\geq\text{dim}\Lambda_{\alpha}\geq\text{dim}\Sigma_{\alpha}.
\end{equation*}
More precisely, the corresponding irreducible components of these varieties satisfy the same inequalities.

\begin{proposition}\label{prop3.2}
\leavevmode
\begin{enumerate}
\item[(i)] Assume that $X$ and $Y$ are as in Case (a).
The varieties $\Theta_{\alpha}$, $\Lambda_{\alpha}$ and $\Sigma_{\alpha}$ are all of pure dimension $n$.
Moreover, no irreducible component of $\Sigma_{\alpha}$ is equal to $\Delta_X$,
and $\Sigma_{\alpha}$ represents an isogenous autocorrespondence of $X$.

\item[(ii)] Assume that $X$ and $Y$ are as in Case (b).
The varieties $\Theta_{\alpha}$, $\Lambda_{\alpha}$ and $\Sigma_{\alpha}$ are all of pure dimension $n+i_X\cdot d_0$.
\end{enumerate}
\end{proposition}
\begin{proof}
The proof is based on a standard deformation-theoretic argument.

(i) Let $X_{0,\alpha}^{[d]}$ (resp. $Y_{0,\alpha}^{[d]}$) be the reduced closed subvariety of $X_0^{[d]}$ (resp. $Y_0^{[d]}$) parameterizing smoothable length-$d$ closed subschemes whose associated cycles are of the form $[Z]=m_1x+m_2y$.  
The varieties $X_{0,\alpha}^{[d]}$ and $Y_{0,\alpha}^{[d]}$ need not be irreducible when $d\geq 4$.
Define the reduced subvariety of $Y_{0,\alpha}^{[d]}\times\text{Hilb}_{R}(Y)$:
\begin{equation*}
\Phi_{\alpha}:=\{(Z,C)\in Y_{0,\alpha}^{[d]}\times\text{Hilb}_{R}(Y)|Z\subseteq C\}.
\end{equation*}
The second projection $\phi_{\alpha,2}: \Phi_{\alpha}\rightarrow\text{Hilb}_{R}(Y)$ is surjective.
For a general curve $C$, its fiber is isomorphic to the stratum $C^{(\alpha)}$ consisting of effective cycles of the form $m_1x+m_2y$. This fiber is irreducible of dimension $2$.
Since $\text{Hilb}_{R}(Y)$ is irreducible of dimension $\widetilde{n}:=d+n-2$,
then the variety $\Phi_{\alpha}$ is irreducible of dimension $\widetilde{n}+2=d+n$. 
Next, define the reduced subvariety of $X_{0,\alpha}^{[d]}\times\text{Hilb}_{R}(Y)$:
\begin{equation*}
\Psi_{\alpha}:=\{(Z,C)\in X_{0,\alpha}^{[d]}\times\text{Hilb}_{R}(Y)|X\cap C\supseteq Z\}.
\end{equation*}
Note that $X_0^{[d]}\times\text{Hilb}_{R}(Y)$ is an integral closed subvariety of $Y_0^{[l]}\times\text{Hilb}_{R}(Y)$ of codimension $d$.
Then at least as sets,
$$
\Psi_{\alpha}=\Phi_{\alpha}\cap (X_{0}^{[d]}\times\text{Hilb}_{R}(Y))
$$
inside $Y_0^{[d]}\times\text{Hilb}_{R}(Y)$. By the projective dimension theorem, $\Psi_{\alpha}$ is nonempty, and
every irreducible component of $\Psi_{\alpha}$ has dimension at least $n$.

An infinitesimal calculation shows that, at a general point, the conditions defining $\Psi_\alpha$ are infinitesimally independent. Hence every irreducible component of $\Psi_\alpha$ has dimension exactly $n$. We do not, however, claim that $\Psi_\alpha$ is irreducible.

There is a natural surjective morphism $\psi_{\alpha}: \Psi_{\alpha}\rightarrow\Theta_{\alpha}$ defined by $\psi_{\alpha}(Z,C)=(x,y, C)$ whenever $[Z]=m_1x+m_2y$.
In particular, $\text{dim}\Psi_{\alpha}\geq\text{dim}\Theta_{\alpha}$.
Let $p_{\alpha,i}:\Sigma_{\alpha}\rightarrow X$ be the natural projections.
We claim that, on every irreducible component of $\Psi_\alpha$, the composite morphisms $\widetilde{p}_{\alpha,i}:=p_{\alpha,i}\circ q_{12}\circ\psi_{\alpha}:\Psi_{\alpha}\rightarrow X$ are surjective.  By symmetry, it suffices to treat $\widetilde{p}_{\alpha,1}$.

Let $Y_{\alpha,c}^{[d]}$ be the reduced variety obtained as the Zariski closure in $Y_{0}^{[d]}$ of the locus 
\begin{equation*}
\{Z_1\sqcup Z_2\in Y_{0}^{[d]}|Z_i\ \text{curvilinear of length}\ m_i\ \text{and supported on some point}\}.
\end{equation*}
Thus, $Y_{\alpha,c}^{[d]}$ parameterizes curvilinear length-$d$ closed subschemes whose associated cycles have the form $m_1y_1+m_2y_2$, together with their deformations in $Y$.
Let $Y_{\alpha}^{(d)}\subseteq Y^{(d)}$ be the closed subvariety consisting of cycles of the form $m_1y_1+m_2y_2$.
It is irreducible of dimension $2(n+1)$.
The variety $Y_{\alpha,c}^{[d]}$ is the irreducible component of $(\pi_0^Y)^{-1}(Y_{\alpha}^{(d)})$ containing curvilinear closed subschemes, where $\pi_0^Y: Y_0^{[d]}\rightarrow Y^{(d)}$ is the restriction of the Hilbert--Chow morphism.
Let $\widetilde{\pi}_{\alpha}^Y: Y_{\alpha,c}^{[d]}\rightarrow Y_{\alpha}^{(d)}$ denote the restriction of $\pi_0^Y$.

Recall that a finite-length subscheme of a smooth variety is smoothable if and only if each of its connected components is smoothable. We emphasize, however, that an arbitrary closed subscheme of a smoothable subscheme need not itself be smoothable.

For a general member $m_1y_1+m_2y_2\in Y_{\alpha}^{(d)}$,
the fiber
\begin{equation*}
(\widetilde{\pi}_{\alpha}^Y)^{-1}(m_1y_1+m_2y_2)=\{Z\sqcup Z_2\in Y_{\alpha,c}^{[d]}|Z_i\in Y_{(m_i),c}^{[m_i]}\}
\end{equation*}
is irreducible of dimension $(m_1-1)n+(m_2-1)n=(d-2)n$.
It follows that $Y_{\alpha,c}^{[d]}$ is irreducible of dimension $dn+2$.
Consequently, its codimension in $Y_{0}^{[d]}$ is exactly $d-2$.

Now fix a closed point $x\in X$. Let $X_{\alpha,x}^{(d)}\subseteq X^{(d)}$ be the reduced closed subvariety consisting of 0-cycles of the form $m_1x+m_2x'$. This variety is irreducible of dimension $n$.
Let $X_{\alpha,c,x}^{[d]}$ be the reduced variety obtained as the Zariski closure in $X_{\alpha,c}^{[d]}$ of the locus
\begin{equation*}
\{Z_1\sqcup Z_2\in X_{\alpha,c}^{[d]}|Z_1\ \text{is supported at}\ x\}.
\end{equation*}
It is the irreducible component of $(\pi_0^X)^{-1}(X_{\alpha,x}^{(d)})$ containing the curvilinear closed subschemes.
Let $\widetilde{\pi}_{\alpha,x}^X: X_{\alpha,c,x}^{[d]}\rightarrow X_{\alpha,x}^{(d)}$ be the restriction of $\pi_0^X$.
Then its general fiber
\begin{equation*}
(\widetilde{\pi}_{\alpha,x}^X)^{-1}(m_1x+m_2x')=\{Z_1\sqcup Z_2\in X_{\alpha,c,x}^{[d]}|Z_2\ \text{supported on}\ x'\}
\end{equation*}
is irreducible of dimension $(d-2)(n-1)$.
Thus $X_{\alpha,c,x}^{[d]}$ is irreducible of dimension $dn+2-d-n$.
Its codimension in $Y_{\alpha,c}^{[d]}$ is thus exactly $d+n$.

Let $\Psi_{\alpha}^k$ be the irreducible components of $\Psi_{\alpha}$, and let $\widetilde{p}_{\alpha,1}^k: \Psi_{\alpha}^k\rightarrow X$ be the restriction of $\widetilde{p}_{\alpha,1}$. 
For any closed point $x\in X$, the fiber of the morphism $\widetilde{p}_{\alpha,1}^k: \Psi_{\alpha}^k\rightarrow X$ over $x$ is (as sets)
\begin{equation*}
(\widetilde{p}_{\alpha,1}^k)^{-1}(x)=\Phi_{\alpha}^k\cap (X_{\alpha,c,x}^{[d]}\times\text{Hilb}_{R}(Y))
\end{equation*}
inside $Y_{\alpha,c}^{[d]}\times\text{Hilb}_{R}(Y)$.
By the projective dimension theorem, this intersection has dimension at least zero and is therefore nonempty.

It follows that $\widetilde p_{\alpha,1}$ is surjective on every irreducible component of $\Psi_\alpha$. By symmetry, the same holds for $\widetilde p_{\alpha,2}$. Consequently, the morphisms
$q_i: \Theta_{\alpha}\rightarrow X$ and $p_{\alpha,i}:\Sigma_{\alpha}\rightarrow X$ are surjective on every irreducible component.

We conclude that the three varieties $\Theta_{\alpha}$, $\Lambda_{\alpha}$ and $\Sigma_{\alpha}$ are all of pure dimension $n$. A general member of $\Lambda_{\alpha}$ is a smooth rational curve.
Moreover, $\Sigma_{\alpha}$ represents an isogenous autocorrespondence of $X$.
Since the cycle of any $Z\in\Psi_{\alpha}$ has the form $[Z]=m_1x+m_2y$ with $x$ and $y$ independent when deforming, then $\Sigma_{\alpha}=q_{12}\psi_{\alpha}(\Psi_{\alpha})$ does not contain $\Delta_X$ as an irreducible component.

(ii) In this case, $\text{Hilb}_{R}(Y)$ is irreducible of dimension $n+i_Y\cdot d_0-2$.
A similar argument as in (i) with minor modifications implies that the three varieties $\Theta_{\alpha}$, $\Lambda_{\alpha}$ and $\Sigma_{\alpha}$ are all of pure dimension $n+i_X\cdot d_0$.
\end{proof}

\begin{remark}
\leavevmode
\begin{enumerate}
\item[(i)] In Case {\rm (a)}, one can verify that $\Sigma_\alpha$ represents a $K$-autocorrespondence. Since this stronger property will not be needed in the sequel, we omit the verification.

\item[(ii)] If $m_1=m_2$, then $\Sigma_{\alpha}$ is symmetric, that is, $\Sigma_{\alpha}=\Sigma_{\alpha}^t$.
\end{enumerate}
\end{remark}

We can now complete the required constructions.
\medskip

\noindent
\emph{Case (a).} Choose an arbitrary irreducible component $\widetilde{\Theta}_{\alpha}\subseteq\Theta_{\alpha}$,
and set $\widetilde{\Lambda}_{\alpha}:=q_3(\widetilde{\Theta}_{\alpha})$ and $\widetilde{\Sigma}_{\alpha}:=q_{12}(\widetilde{\Theta}_{\alpha})$.
\medskip

\noindent
\emph{Case (b).} Let $\eta$ denote the generic point of $X$.
The fibers $p_{\alpha,1}^{-1}(\eta)=p_{\alpha,2}^{-1}(\eta)$ and $q_{1}^{-1}(\eta)=q_{2}^{-1}(\eta)$
are all of pure dimension $i_X\cdot d_0$.
The induced morphism $q_{1}^{-1}(\eta)\rightarrow p_{\alpha,1}^{-1}(\eta)$ is generically finite and surjective.
Choose a closed point $\xi$ of $q_{1}^{-1}(\eta)$ such that the Zariski closure of $q_{12}(\xi)$ is not equal to $\Delta_X$. Let $\widetilde{\Theta}_\alpha$ be the Zariski closure of $\xi$ in $\Theta_\alpha$, let $\widetilde{\Lambda}_\alpha$ be the Zariski closure of $q_3(\xi)$ in $\operatorname{Hilb}_R(Y)$, and let $\widetilde{\Sigma}_{\alpha}$ be the Zariski closure of $q_{12}(\xi)$ in $\Sigma_{\alpha}$.

In both cases, the varieties $\widetilde{\Theta}_{\alpha}$, $\widetilde{\Lambda}_{\alpha}$ and $\widetilde{\Sigma}_{\alpha}$ are integral of dimension $n$. 
Moreover, $\widetilde{\Sigma}_{\alpha}\neq\Delta_X$ and $\widetilde{\Sigma}_{\alpha}$ is an isogenous autocorrespondence of $X$.
Define two reduced varieties:
\begin{equation*}
\Pi_{\alpha}:=\{(x, y)\in Y^2|x, y\in C,\ [C]\in\widetilde{\Lambda}_{\alpha}\}
\end{equation*}
and
\begin{equation*}
\Xi_{\alpha}:=\{(x, y, z)\in Y^3|x, y, z\in C,\ [C]\in\widetilde{\Lambda}_{\alpha}\}.
\end{equation*}
Then $\Pi_{\alpha}$ and $\Xi_{\alpha}$ are integral of dimension $n+2$ and $n+3$, respectively.

These varieties naturally give rise to cycle relations modulo rational equivalence.
\begin{proposition}\label{prop3.4}
Let $X$ and $Y$ be as in Case (a) or Case (b).
Let $\delta_{12}: X^2\rightarrow X^3$ be the morphism defined by $(x,y)\mapsto (x, x, y)$, 
and define $\delta_{13}$ and $\delta_{23}$ similarly.
\leavevmode
\begin{enumerate}
\item[(i)] There exists an equality of cycle classes:
\begin{equation}\label{34}
(\iota\times\iota)^*\Pi_{\alpha}=a_1\Delta_X+a_2(\widetilde{\Sigma}_{\alpha}+\widetilde{\Sigma}_{\alpha}^t)
\end{equation}
in $\mathrm{CH}^n(X^2)$, where $a_1, a_2$ are positive integers.

\item[(ii)] There exists an equality of cycle classes:
\begin{equation}\label{35}
(\iota\times\iota\times\iota)^*\Xi_{\alpha}
=b_1\Delta_{123}+b_2[\delta_{12*}(\widetilde{\Sigma}_{\alpha}+\widetilde{\Sigma}_{\alpha}^t)
+\delta_{13*}(\widetilde{\Sigma}_{\alpha}+\widetilde{\Sigma}_{\alpha}^t)+\delta_{23*}(\widetilde{\Sigma}_{\alpha}
+\widetilde{\Sigma}_{\alpha}^t)]
\end{equation}
in $\mathrm{CH}^{2n}(X^3)$, where $b_1, b_2$ are positive integers.

\item[(iii)] Assume that the cycle class $[\Xi_{\alpha}]\in A^{2n}(Y^3):=\mathrm{CH}^{2n}(Y^3)/\sim_{\text{hom}}$ is a $\mathbb{Q}$-linear combination of intersections of big diagonal classes and completely decomposable cycle classes. (E.g., $Y$ has trivial Chow groups).
Then 
\begin{equation}\label{36}
\frac{2a_1}{a_2}=\frac{b_1}{b_2}+d_1+d_2,
\end{equation}
where $d_i$ denotes the degree of the natural projection of $\widetilde{\Sigma}_{\alpha}$ onto the $i$-th factor.
\end{enumerate}
\end{proposition}
\begin{proof}
(i) Observe that as sets, one has
\begin{equation*}
\Pi_{\alpha}\cap X^2=\Delta_X\cup\widetilde{\Sigma}_{\alpha}\cup\widetilde{\Sigma}_{\alpha}^t.
\end{equation*}
Thus the two varieties $\Pi_{\alpha}$ and $X^2$ intersect properly in $Y^2$.
By intersection theory, one has 
\begin{equation*}
(\iota\times\iota)^*\Pi_{\alpha}=a_1\Delta_X+a_2\widetilde{\Sigma}_{\alpha}+a'_2\widetilde{\Sigma}_{\alpha}^t
\end{equation*}
in $\mathrm{CH}^n(X^2)$ for some positive integers $a_1, a_2, a'_2$ by Proposition 8.2 in \cite{Fu98}.
Since $\Pi_{\alpha}$ is symmetric by construction, one has $a_2=a'_2$, which proves \eqref{34}.

(ii) The two varieties $\Xi_{\alpha}$ and $X^3$ intersect properly inside $Y^3$. 
Set-theoretically,
\begin{equation*}
\Xi_{\alpha}\cap X^3
=\Delta_{123}^X\cup(\cup_{i<j}\delta_{ij*}(\widetilde{\Sigma}_{\alpha}))
\cup(\cup_{i<j}\delta_{ij*}(\widetilde{\Sigma}_{\alpha}^t)).
\end{equation*}
It follows that
\begin{equation*}
(\iota\times\iota\times\iota)^*\Xi_{\alpha}
=b_1\Delta_{123}^X+b_2\delta_{12*}(\widetilde{\Sigma}_{\alpha}+\widetilde{\Sigma}_{\alpha}^t)
+b'_2\delta_{13*}(\widetilde{\Sigma}_{\alpha}+\widetilde{\Sigma}_{\alpha}^t)
+b''_2\delta_{23*}(\widetilde{\Sigma}_{\alpha}+\widetilde{\Sigma}_{\alpha}^t)
\end{equation*}
in $\mathrm{CH}^{2n}(X^3)$ for some positive integers $b_1, b_2, b'_2,b''_2$.
Since $\Xi_{\alpha}$ is invariant under permutations of the three factors, one has $b_2=b'_2=b''_2$. 
This proves \eqref{35}.

(iii) Applying the pushforward $p_{12*}$ to \eqref{35}, we obtain
\begin{equation}\label{37}
p_{12*}(\iota\times\iota\times\iota)^*\Xi_{\alpha}
=(b_1+b_2(d_1+d_2))\Delta_X +2b_2(\widetilde{\Sigma}_{\alpha}+\widetilde{\Sigma}_{\alpha}^t).
\end{equation}
Modulo homological equivalence for \eqref{34} and \eqref{37}, one gets the following two equalities in $A^n(X^2)$:
\begin{equation}\label{38}
a_1[\Delta_X]+a_2([\widetilde{\Sigma}_{\alpha}]+[\widetilde{\Sigma}_{\alpha}^t])=[(\iota\times\iota)^*\Pi_{\alpha}]
\end{equation}
and
\begin{equation}\label{39}
(b_1+b_2(d_1+d_2))[\Delta_X]+2b_2([\widetilde{\Sigma}_{\alpha}]+[\widetilde{\Sigma}_{\alpha}^t])
=[p_{12*}(\iota\times\iota\times\iota)^*\Xi_{\alpha}].
\end{equation}
Substituting \eqref{38} into \eqref{39}, one has
\begin{equation}\label{40}
\left(b_1+b_2(d_1+d_2)-\frac{2a_1b_2}{a_2}\right)[\Delta_X]
=-\frac{1}{a_2}[(\iota\times\iota)^*\Pi_{\alpha}]+[p_{12*}(\iota\times\iota\times\iota)^*\Xi_{\alpha}].
\end{equation}

Since the cycle class $[\Xi_{\alpha}]$ is symmetric, the assumption allows us to express it as a $\mathbb{Q}$-linear combination of completely decomposable classes and symmetrized classes of the form
\begin{equation*}
z:=p_{12}^*\delta_{Y*}z_1\cdot p_3^*z_2+p_{13}^*\delta_{Y*}z_1\cdot p_2^*z_2+p_{23}^*\delta_{Y*}z_1\cdot p_1^*z_2.
\end{equation*}
Now since the pushforward $p_{12*}z$ is completely decomposable,
then so is the cycle class $[p_{12*}(\iota\times\iota\times\iota)^*\Xi_{\alpha}]$.
Let $\widetilde{\pi}_n^{\text{van}}$ denote the projector corresponding to the vanishing cohomology $H^n(X,\mathbb{Q})_{\text{van}}\subseteq H^n(X,\mathbb{Q})$.
Applying $\widetilde{\pi}_n^{\text{van}}\otimes\widetilde{\pi}_n^{\text{van}}$ to \eqref{40},
one has
\begin{equation}\label{41}
\left(b_1+b_2(d_1+d_2)-\frac{2a_1b_2}{a_2}\right)[\Delta_X^{\text{van}}]=0,
\end{equation}
where $[\Delta_X^{\text{van}}]:=(\widetilde{\pi}_n^{\text{van}}\otimes\widetilde{\pi}_n^{\text{van}})_*[\Delta_X]$.
By assumption, $\iota^*: H^n(X,\mathbb{Q})\rightarrow H^n(Y,\mathbb{Q})$ is not isomorphic, 
then $[\Delta_X]$ cannot be the restriction of a cycle class on $Y^2$. 
So, $[\Delta_X^{\text{van}}]\neq 0$.
By \eqref{41}, this forces that
$$
b_1+b_2(d_1+d_2)-\frac{2a_1b_2}{a_2}=0.
$$
Equivalently, 
$$
\frac{2a_1}{a_2}=\frac{b_1}{b_2}+d_1+d_2.
$$
\end{proof}
\begin{remark}
The method used above was first introduced for K3 surfaces in \cite{Ba19}.
\end{remark}
\begin{corollary}\label{cor3.6}
Assume that one of the following two conditions holds:
\begin{enumerate}
\item[(i)] The cycle classes $\Pi_{\alpha}\in\mathrm{CH}^n(Y^2)$ and $\Xi_{\alpha}\in\mathrm{CH}^{2n}(Y^3)$ are completely decomposable. This holds, for example, if $Y$ has trivial Chow groups.

\item[(ii)] The variety $Y$ has Picard number one, the cycle class $\Pi_{\alpha}\in\mathrm{CH}^n(Y^2)$ is completely decomposable, and $\Xi_{\alpha}\in\mathrm{CH}^{2n}(Y^3)$ is a $\mathbb{Q}$-linear combination of intersections of big diagonal classes and divisor classes. This holds, for example, when $Y$ is a Fano complete intersection; see Theorem \ref{thm4.19}.
\end{enumerate}
Then there is an equality of cycle classes:
\begin{equation}\label{42}
\Delta_{123}=\Delta_{12}+\Delta_{13}+\Delta_{23}
+\delta_{12*}z+\delta_{13*}z+\delta_{23*}z+(\iota\times\iota\times\iota)^*w
\end{equation}
for some completely decomposable cycle classes $z\in\mathrm{CH}^n(X^2)$ and $w\in\mathrm{CH}^{2n}(Y^3)$.
\end{corollary}
\begin{proof}
By assumption, one can write $(\iota\times\iota)^*\Pi_{\alpha}=c'(X\times o_X)+z'$, where $z'$ is a $\mathbb{Q}$-linear combination of cycle classes of the form $z_1\times z_2$, $z_1\in\mathrm{CH}^i(X)$, $z_2\in\mathrm{CH}^{n-i}(X)$, $i>0$.
By \eqref{34} and \eqref{35} of Proposition \ref{prop3.4}(i)(ii), it is immediate to get that
\begin{equation}\label{43}
(b_1-\frac{3a_1b_2}{a_2})\Delta_{123}
-c(\Delta_{12}+\Delta_{13}+\Delta_{23})
=\frac{b_2}{a_2}\left(\delta_{12*}z'+\delta_{13*}z'
+\delta_{23*}z'\right)-(\iota\times\iota\times\iota)^*\Xi_{\alpha},
\end{equation}
where $c=\frac{b_2c'}{a_2}$.
By \eqref{36} of Proposition \ref{prop3.4} (iii), one has
\begin{equation*}
b_1-\frac{3a_1b_2}{a_2}=-\left(\frac{a_1b_2}{a_2}+b_2(d_1+d_2)\right)\neq 0.
\end{equation*}
Note that if $\mathrm{CH}^1(Y)=\mathbb{Q}\cdot H$, then
$\Delta_X\cdot\iota^*H\times X=\frac{1}{d}(\iota\times\iota)^*\Delta_Y$, where $[X]=dH$.
Applying the pushforward $p_{12*}$ to \eqref{43}, under assumption either (i) or (ii), 
the cycle class in the right-hand side of \eqref{43} becomes completely decomposable.
Then this gives that $c=b_1-\frac{3a_1b_2}{a_2}$.

Dividing \eqref{43} by this nonzero constant and absorbing the completely decomposable terms into a cycle class $z\in\mathrm{CH}^n(X^2)$, we obtain
\begin{equation*}
\Delta_{123}
=\Delta_{12}+\Delta_{13}+\Delta_{23}+\delta_{12*}z+\delta_{13*}z+\delta_{23*}z+(\iota\times\iota\times\iota)^*w
\end{equation*}
for some completely decomposable cycle class $w\in\mathrm{CH}^{2n}(Y^3)$.
\end{proof}
\bigskip
\bigskip

\section{Complete intersections}
\bigskip

In this section, we focus primarily on complete intersections in certain ambient varieties.
The main results are Theorem \ref{thm4.3}, Theorem \ref{thm4.11} and Theorem \ref{thm4.19}.

We begin with the general setting.
Let $Y$ be a smooth connected projective variety with trivial Chow groups, and let $X\subseteq Y$ be a smooth connected \emph{ample} subvariety, for instance, a smooth complete intersection. We refer to \cite{Ot12} for the definition of an ample subvariety.
Set $N=\text{dim}Y$, $n=\text{dim}X$, $e:=N-n\geq 1$, and let $\iota: X\rightarrow Y$ denote the inclusion.
By \cite{Ot12}, the restriction map $\iota^*: H^{k}(Y,\mathbb{Q})\rightarrow H^{k}(X,\mathbb{Q})$ 
is an isomorphism for every integer $k<n$ and is injective for $k=n$,
whereas the Gysin pushforward $\iota_*: H^{k}(X,\mathbb{Q})\rightarrow H^{k+2e}(Y,\mathbb{Q})$ is an isomorphism for every integer $k>n$.
Let $b_i$ denote the $i$-th Betti number of $Y$.
For every integer $i$, choose a Poincaré dual basis $\{\beta_{ij}\}_{1\leq j\leq b_{2i}}$ of $\mathrm{CH}^i(Y)\cong H^{2i}(Y,\mathbb{Q})$ such that $\beta_{01}=[Y]$ and $o_Y:=\beta_{N,1}$ is a degree-one 0-cycle class.
Thus $\int_Y \beta_{ij}\cdot \beta_{N-i,j'}=\delta_{jj'}$.
Now fix an ample divisor class $L\in\mathrm{CH}^1(Y)$.
By the hard Lefschetz theorem for singular cohomology,
there exist cycle classes $w_{ij}\in L^{2i-n}\cdot\mathrm{CH}^{n-i}(Y)\subseteq\mathrm{CH}^{i}(Y)$ such that $\iota_*\iota^*w_{ij}=\beta_{i+e,j}$ for every integer $\frac{n}{2}\leq i\leq n$ and every $j$.
Here $\iota_*:\mathrm{CH}_*(X)\rightarrow\mathrm{CH}_*(Y)$ and $\iota^*:\mathrm{CH}^*(Y)\rightarrow\mathrm{CH}^*(X)$
denote the proper pushforward and the pullback, respectively.
For every integer $0\leq i\leq n$ with $i\neq\frac{n}{2}$, define the cycle classes $\alpha_{ij}\in\mathrm{CH}^i(X)$ by
\begin{equation*}
\alpha_{ij}:=\left\{
\begin{array}{ccc} \iota^*\beta_{ij}, && \text{if}\ \ i<\frac{n}{2};\\
\iota^*w_{ij}, && \text{if}\ \ i>\frac{n}{2}.
\end{array}
\right.
\end{equation*}
Then one has $\int_X \alpha_{ij}\cdot \alpha_{n-i,j'}=\delta_{jj'}$.
Moreover, $\{cl_X(\alpha_{ij})\}_{1\leq j\leq b_{2i}}$ forms a basis of $H^{2i}(X,\mathbb{Q})$ for every $i\neq \frac{n}{2}$, 
where $cl_X: \mathrm{CH}^i(X)\rightarrow H^{2i}(X,\mathbb{Q})$ is the cycle class map.
In particular, $\alpha_{01}=[X]$ and $o_X:=\alpha_{n1}=\frac{1}{\int_X\iota^*L^n}\iota^*L^n$ is a degree-one 0-cycle class.
For every integer $0\leq k\leq 2n$, define
\begin{equation*}
\pi_k^X:=\left\{
\begin{array}{ccc} \sum_{j=1}^{b_{k}}\alpha_{n-i,j}\times\alpha_{ij}, && \text{if}\ k=2i<n;\\
\sum_{j=1}^{b_{k+2e}}\alpha_{n-i,j}\times\alpha_{ij}, && \text{if}\ k=2i>n;\\
0, && \quad\text{if}\ \ k\neq n\ \text{is odd};\\
\Delta_X-\sum_{\ell\neq n}\pi_{\ell}^X, && \text{if}\ \ k=n.
\end{array}
\right.
\end{equation*}
When $n=2i$ is even, choose a basis $\{\alpha_{ij}=\iota^*\gamma_{ij}\}_{1\leq j\leq b_n}$ of $\iota^*\mathrm{CH}^i(Y)\subseteq\mathrm{CH}^i(X)$
such that their cohomology classes forms an orthogonal basis of the ambient part $\iota^*H^n(Y,\mathbb{Q})\subseteq H^n(X,\mathbb{Q})$.
Define $\pi_n^{a}:=\sum_{j=1}^{b_n}\alpha_{ij}\times\alpha_{ij}$ and $\pi_n^{\text{van}}:=\pi_n-\pi_n^{a}$.
Note that $\pi_0^X=o_X\times X$ and $\pi_{2n}^X=X\times o_X$.
It is straightforward to verify that $\{\pi_k^X\}_{0\leq k\leq 2n}$ forms a set of self-dual CK projectors of $X$.
Throughout this section, $X$ will always be equipped with this \emph{natural} self-dual CK decomposition.

For the study of motivic $0$-multiplicativity, it is useful to reformulate the condition as a relation among cycle classes.
\begin{lemma}\label{lem4.1}
The following conditions are equivalent.
\begin{enumerate}
  \item [(i)] The natural CK decomposition of $X$ is multiplicative.
  \item [(ii)] There is an equality of cycles classes:
\begin{equation}\label{44}
\Delta_{123}-(\Delta_{12}+\Delta_{13}+\Delta_{23})=\Omega_X
\end{equation}
in $\mathrm{CH}^{2n}(X^3)$.
Here for $\{i,j,k\}=\{1, 2, 3\}$, $\Delta_{ij}:=p_{ij}^*\Delta_X\cdot p_k^*o_X$,
$\Delta_i:=p_i^*[X]\cdot p_{j}^*o_X\cdot p_k^*o_X$ 
and

\begin{eqnarray*}
 \Omega_X:  &=& \sum_{i,j>0; i\neq\frac{n}{2},j\neq\frac{n}{2}; i+j\notin\{n,\frac{3n}{2}\}}
\sum_{l_1, l_2, l_3}a_{ij}^{l_1l_2l_3}(\alpha_{n-i, l_1}\times\alpha_{n-j, l_2}\times\alpha_{i+j, l_3}) \\
   &+& \sum_{i\neq\frac{n}{2},n}\sum_{l_1, l_2}\left[(\alpha_{n-i, l_1}\cdot\alpha_{i-\frac{n}{2},l_2})\times\alpha_{il_1}\times\alpha_{\frac{3n}{2}-i,l_2}
   +\alpha_{il_1}\times(\alpha_{n-i, l_1}\cdot\alpha_{i-\frac{n}{2},l_2})\times\alpha_{\frac{3n}{2}-i,l_2}\right] \\
   &+& \sum_{\frac{n}{2}<i<\frac{3n}{2},i\neq n}\sum_{l_1, l_2}\alpha_{il_1}\times\alpha_{\frac{3n}{2}-i,l_2}\times(\alpha_{n-i, l_1}\cdot\alpha_{i-\frac{n}{2},l_2}) \\
   &+&\sum_{i\neq\frac{n}{2}}\sum_{l_1, l_2}2a_{n-i}^{l_1l_2}(o_X\times\alpha_{i,l_1}\times\alpha_{n-i,l_2} +\alpha_{i,l_1}\times o_X\times\alpha_{n-i,l_2}+\alpha_{i,l_1}\times\alpha_{n-i,l_2}\times o_X)
\end{eqnarray*}
where $a_{i}^{l_1l_2}:=\int_X\alpha_{il_1}\cdot\alpha_{n-i,l_2}\in\mathbb{Q}$ and $a_{ij}^{l_1l_2l_3}:=\int_X\alpha_{il_1}\cdot\alpha_{jl_2}\cdot\alpha_{n-i-j,l_3}\in\mathbb{Q}$.
\end{enumerate}
\end{lemma}
\begin{proof}
Let $i,j\neq\frac{n}{2}$.
If $i+j\geq n$ and $i+j\neq\frac{3n}{2}$, then one has
\begin{equation*}
(\pi_{2i}\otimes\pi_{2j}\otimes\pi_{4n-2i-2j})_*\Delta_{123}
=\sum_{l_1, l_2, l_3}a_{n-i,n-j}^{l_1l_2l_3}(\alpha_{il_1}\times\alpha_{jl_2}\times\alpha_{2n-i-j,l_3}).
\end{equation*}
If $i+j=\frac{3n}{2}$, then
\begin{equation*}
(\pi_{2i}\otimes\pi_{2j}\otimes\pi_n)_*\Delta_{123}
=\sum_{l_1, l_2}\alpha_{il_1}\times\alpha_{jl_2}\times(\alpha_{n-i, l_1}\cdot\alpha_{n-j,l_2}).
\end{equation*}
Moreover,
\begin{equation*}
(\pi_n\otimes\pi_n\otimes\pi_{2n})_*\Delta_{123}
=\Delta_{12}+\sum_{l_1, l_2}\sum_{i\neq\frac{n}{2}}a_{n-i}^{l_1l_2}(\alpha_{i,l_1}\times\alpha_{n-i,l_2}\times o_X).
\end{equation*}
Then the result follows immediately by expanding out the equality 
$$
\Delta_{123}=\sum_{k=0}^{2n}\sum_{i+j=k}(\pi_{2n-i}\otimes\pi_{2n-j}\otimes\pi_k)_*\Delta_{123}.
$$
\end{proof}
\begin{remark}
As a consequence, motivic 0-multiplicativity is preserved under specialization for varieties of this kind,
since the cycle relation \eqref{44} is generically defined.
\end{remark}

The first main result of this section is the following.
\begin{theorem}\label{thm4.3}
Let $Y$ be a smooth connected projective variety with trivial Chow groups,
and let $X\subseteq Y$ be a smooth connected ample subvariety. 
Then $\mathfrak{d}_m(X)\leq mn$ for every integer $m\geq 2$.
\end{theorem}
\begin{proof}
We begin with the case $m=2$.
By straightforward calculations, for three integers $i, j, k\neq\frac{n}{2}$, one gets that
\begin{eqnarray*}
&  & \pi_{2k}\circ\Delta_{123}\circ (\pi_{2i}\otimes\pi_{2j})\\
&=& \sum_{l}\sum_{l'}\sum_{l''}\left(\int_X\alpha_{il}\cdot\alpha_{jl'}\cdot\alpha_{n-k,l''}\right)
(\alpha_{n-i, l}\times\alpha_{n-j, l'}\times\alpha_{kl''}),
\end{eqnarray*}
\begin{equation*}
\pi_{2k}\circ\Delta_{123}\circ (\Delta_X\otimes\pi_{2j})
=\sum_{l}\sum_{l'}(\alpha_{jl}\cdot\alpha_{n-k,l'})\times\alpha_{n-j,l}\times\alpha_{kl'},
\end{equation*}
\begin{equation*}
\pi_{2k}\circ\Delta_{123}\circ (\pi_{2i}\otimes\Delta_X)
=\sum_{l}\sum_{l'}\alpha_{n-i,l}\times(\alpha_{il}\cdot\alpha_{n-k,l'})\times\alpha_{kl'},
\end{equation*}
\begin{equation*}
\Delta_{123}\circ (\pi_{2i}\otimes\pi_{2j})
=\sum_{l}\sum_{l'}\alpha_{n-i,l}\times\alpha_{n-j,l'}\times(\alpha_{i,l}\cdot\alpha_{j,l'}),
\end{equation*}
\begin{equation*}
\pi_{2k}\circ\Delta_{123}
=(\Delta_X\otimes\Delta_X\otimes\pi_{2k})_*\Delta_{123}
=\sum_{l}\delta_{X*}\alpha_{n-k,l}\times\alpha_{kl},
\end{equation*}
\begin{equation*}
\Delta_{123}\circ (\Delta_X\otimes\pi_{2j})
=\sum_{l}p_{13}^*(\delta_{X*}\alpha_{jl})\cdot p_2^*\alpha_{n-j,l},
\end{equation*}
\begin{equation*}
\Delta_{123}\circ (\pi_{2i}\otimes\Delta_X)
=\sum_{l}\alpha_{n-i,l}\times\delta_{X*}\alpha_{il},
\end{equation*}
\begin{equation*}
\pi_0\circ\Delta_{123}\circ(\pi_n\otimes\pi_n)=(\pi_n\otimes\pi_n)_*(\delta_{X*}o_X)\times X=0,
\end{equation*}
\begin{equation*}
\pi_{n}\circ\Delta_{123}\circ(\pi_n\otimes\pi_{2n})=0,
\end{equation*}
\begin{equation*}
\pi_{n}\circ\Delta_{123}\circ(\pi_{2n}\otimes\pi_{n})=0.
\end{equation*}
If each $i, j, k\neq\frac{n}{2}$ and $k\neq i+j$, then $\int_X\alpha_{il}\cdot\alpha_{jl'}\cdot\alpha_{n-k,l''}=0$ for dimension reasons and thus
\begin{equation*}
\pi_{2k}\circ\Delta_{123}\circ (\pi_{2i}\otimes\pi_{2j})=0.
\end{equation*}
For $i\neq\frac{n}{2}$ and each $l$, one has
\begin{eqnarray*}
\pi_{n*}\alpha_{il}
&=& (\Delta_X-\sum_{j\neq\frac{n}{2}}\sum_{l'}\alpha_{n-j,l'}\times\alpha_{jl'})_*\alpha_{il}\\
&=& \alpha_{il}-\sum_{j\neq\frac{n}{2}}\sum_{l'}\left(\int_X\alpha_{n-j,l'}\cdot\alpha_{il}\right)\alpha_{jl'}\\
&=& 0.
\end{eqnarray*}
Thus, if $k>\frac{n}{2}+i$, then $n+i-k<\frac{n}{2}$ and hence the cycle class $\alpha_{il}\cdot\alpha_{n-k,l'}$ is a $\mathbb{Q}$-linear combination of $\alpha_{n+i-k,l''}$, according to Lemma \ref{lem4.8}.
If $k<i$, then $n+i-k>n$ and hence $\alpha_{il}\cdot\alpha_{n-k,l'}=0$ for dimension reasons.
So, for both cases one gets that
\begin{equation*}
\pi_{2k}\circ\Delta_{123}\circ (\pi_n\otimes\pi_{2i})
=\sum_{l}\sum_{l'}\pi_{n*}(\alpha_{il}\cdot\alpha_{n-k,l'})\times\alpha_{n-i,l}\times\alpha_{kl'}=0
\end{equation*}
and
\begin{equation*}
\pi_{2k}\circ\Delta_{123}\circ (\pi_{2i}\otimes\pi_n)
=\sum_{l}\sum_{l'}\alpha_{n-i,l}\times\pi_{n*}(\alpha_{il}\cdot\alpha_{n-k,l'})\times\alpha_{kl'}=0.
\end{equation*}

If $i+j<\frac{n}{2}$ or $i+j>n$, then
\begin{eqnarray*}
\pi_n\circ\Delta_{123}\circ (\pi_{2i}\otimes\pi_{2j})
&=& (\pi_{2n-2i}\otimes\pi_{2n-2j}\otimes\pi_n)_*\Delta_{123}\\
&=& \sum_{l}\sum_{l'}\alpha_{n-i,l}\times\alpha_{n-j,l'}\times \pi_{n*}(\alpha_{il}\cdot\alpha_{jl'}))\\
&=& 0.
\end{eqnarray*}
Therefore, we obtain that $\mathfrak{d}(X)\leq 2n-1$.

We now consider the case $m\geq 3$. 
Suppose first that $k>\sum_{\ell=1}^{m}i_{\ell}$.
Then at least one $i_{\ell'}$ is different from $n$; 
otherwise, $k>mn\geq 3n$, which is impossible.
By symmetry, we may assume that $i_1=2j$ and $j\neq\frac{n}{2}$ and $j\leq n$. 
Then
\begin{eqnarray*}
   \pi_k\circ\Delta_I\circ(\pi_{i_1}\otimes\cdots\otimes\pi_{i_m})
   &=& (\pi_{2n-i_1}\otimes\cdots\otimes\pi_{2n-i_m}\otimes\pi_k)_*\Delta_I \\
   &=&
   \sum_{l}(\Delta_X\otimes\pi_{2n-i_2}\otimes\cdots\otimes\pi_{2n-i_m}\otimes\pi_k)_*
   (\alpha_{n-j,l}\times\delta_{m*}\alpha_{jl})\\ 
   &=& \sum_{l}\alpha_{n-j,l}\times(\pi_{2n-i_2}\otimes\cdots\otimes\pi_{2n-i_m}\otimes\pi_k)_*(\delta_{m*}\alpha_{jl}).
\end{eqnarray*}
It is therefore enough to prove the following assertion: for every $m\geq 1$, every $j\neq\frac{n}{2}$, every $\ell$ and 
every collection of indices satisfying $\sum_{\ell=1}^{m}k_{\ell}>2(m-1)n+2j$, one has
\begin{equation}\label{45}
(\pi_{k_1}\otimes\cdots\otimes\pi_{k_m})_*(\delta_{m*}\beta)=0
\end{equation}
for any cycle class $\beta\in\mathrm{CH}^j(X)$.

We argue by induction on $m\geq 1$. The case $m=1$ is immediate.
Assume that $m\geq 2$. At least one $k_\ell$ is different from $n$. 
By symmetry, we may assume that $k_1=2n-2j'$, $j'\neq\frac{n}{2}$.
Then
\begin{equation*}
(\pi_{k_1}\otimes\cdots\otimes\pi_{k_m})_*(\delta_{m*}\beta)
=\sum_{l'}\alpha_{n-j',l'}\times (\pi_{k_2}\otimes\cdots\otimes\pi_{k_m})_*((\delta_{m-1})_*(\beta\cdot\alpha_{j'l'})).
\end{equation*}
The latter vanishes by induction, since $\sum_{\ell=2}^{m}k_{\ell}>2(m-2)n+2j+2j'$ and $\beta\cdot\alpha_{j'l'}\in\mathrm{CH}^{j+j'}(X)$.

Suppose that $k<\sum_{\ell=1}^{m}i_{\ell}-mn$.
Then $i_{\ell'}>n$ for some $\ell'$.
We may assume that $i_1=2j>n$ by symmetry.
Then
\begin{equation*}
\pi_k\circ\Delta_I\circ(\pi_{i_1}\otimes\cdots\otimes\pi_{i_m})
=\sum_{l}\alpha_{n-j,l}\times(\pi_{2n-i_2}\otimes\cdots\otimes\pi_{2n-i_m}\otimes\pi_k)_*(\delta_{m*}\alpha_{jl}).
\end{equation*}
So it suffices to show that 
\begin{equation}\label{46}
(\pi_{2n-i_2}\otimes\cdots\otimes\pi_{2n-i_m}\otimes\pi_k)_*(\delta_{m*}\beta)=0
\end{equation}
for any cycle class $\beta\in\mathrm{CH}^j(X)$.

Again, we argue by induction on $m$. When $m=1$, the assertion reduces to
$\pi_{k*}\beta=0$, which follows from the inequality $k<j$.
Now assume that $m\geq 2$. Then one has $\sum_{\ell=2}^{m}2n-i_{\ell}+k<(m-1)n+j$.
Hence either $i_{\ell'}\neq n$ for some $\ell'\geq 2$, or $k\neq n$.
If $i_{\ell'}\neq n$, then after permuting the factors, we may assume that $i_2=2j'<n$.
By induction,
\begin{eqnarray*}
   & & (\pi_{2n-i_2}\otimes\cdots\otimes\pi_{2n-i_m}\otimes\pi_k)_*(\delta_{m*}\beta) \\
   &=& \sum_{l}\alpha_{n-j',l}\times(\pi_{2n-i_3}\otimes\cdots\otimes\pi_{2n-i_m}\otimes\pi_k)_*
   ((\delta_{m-1})_*(\beta\cdot\alpha_{j'l})) \\
   &=& 0,
\end{eqnarray*}
since $\sum_{\ell=3}^{m}2n-i_{\ell}+k<(m-2)n+j+j'$.
If $k=2k'<n$, then $j+n-k'>n$ and hence
\begin{eqnarray*}
   & &  (\pi_{2n-i_2}\otimes\cdots\otimes\pi_{2n-i_m}\otimes\pi_k)_*(\delta_{m*}\beta) \\
   &=& \sum_{l}(\pi_{2n-i_2}\otimes\cdots\otimes\pi_{2n-i_m})_*((\delta_{m-1})_*(\beta\cdot\alpha_{n-k',l}))\times\alpha_{k'l} \\
   &=& 0.
\end{eqnarray*}
This completes the proof.
\end{proof}

\begin{corollary}
The inequality $\mathfrak{d}_m(X)\leq mn$ is preserved under arbitrary products of curves, surfaces, and smooth ample subvarieties in ambient varieties with trivial Chow groups.
\end{corollary}

For complete intersections in projective spaces, one can obtain a slightly sharper bound for the motivic multiple multiplicativity defect.
\begin{proposition}\label{prop4.5}
Let $Y=\mathbb{P}^{n+e}$ and let $X\subseteq Y$ be a smooth complete intersection of dimension $n\geq 2$.
Then $\mathfrak{d}_m(X)\leq mn-1$ for every integer $m\geq 2$.
If $e\leq n$, then 
\begin{equation}\label{47}
\mathfrak{d}_m(X)\leq \text{max}\{(m-1)n, (m-2)n+2e-1\}.
\end{equation}
Let $D\in\mathrm{CH}^1(X)$ denote the hyperplane class. 
If the cycle class $\delta_{X*}D^k$ is completely decomposable for every integer $2\leq k\leq\frac{n}{2}+1$, 
then $\mathfrak{d}(X)\leq n$.
For every integer $\frac{n}{2}<k\leq n-1$, the natural CK decomposition of $X$ is strictly $2k$-multiplicative if and only if $\delta_{X*}D^{k+1}$ is completely decomposable, whereas $\delta_{X*}D^{k}$ is not.
\end{proposition}
\begin{proof}
Let $d:=\int_XD^n$ be the degree of $X$. Then $o_X=\frac{1}{d}D^n$.
For every integer $0\leq i\leq n$ with $i\neq\frac{n}{2}$, by definition, one has
\begin{equation*}
\pi_{2i}:=\frac{1}{d}D^{n-i}\times D^i,\quad
\pi_n:=\Delta_X-\sum_{i\neq\frac{n}{2}}\pi_{2i}.
\end{equation*}
\medskip

\emph{The case $m=2$.}
If $e\leq n$, then it follows easily that 
\begin{equation*}
\delta_{X*}D^e=\Delta_X\cdot (D^e\times X)
=\frac{1}{d}(\iota\times\iota)^*\Delta_{\mathbb{P}^{n+e}}=\frac{1}{d}\sum_{j=e}^{n+e}D^{n+e-j}\times D^j.
\end{equation*}
For every integer $m\geq2$ and $i\geq e$, one obtains
\begin{equation*}
\delta_{X*}D^i=\Delta_X\cdot (D^i\times X)=\frac{1}{d}\sum_{j=i}^{n+e}D^{n+i-j}\times D^j.
\end{equation*}
Thus, for every $m\geq 2$ and $i\geq e$, one gets that
\begin{eqnarray}
\delta_{m*}D^i
&=& \Delta_{I_{m-1}}\cdot (D^i\times D^{m-1})\label{48}\\
&=& (\Delta_X\times X^{m-2})\cdot (D^i\times X^{m-1})\cdot (X\times\Delta_{I_{m-1}})\nonumber \\
&=& \frac{1}{d}\sum_{j=i}^{n+e}D^{n+i-j}\times\delta_{m-1*}D^j.\nonumber
\end{eqnarray}

If $k\leq n-e$, then 
\begin{eqnarray*}
\pi_{2k}\circ\Delta_{123}\circ(\pi_n\otimes\pi_n)
&=& (\pi_n\otimes\pi_n\otimes\pi_{2k})_*(\Delta_{123}) \\
&=& \frac{1}{d}(\pi_n\otimes\pi_n)_*(\delta_{X*}D^{n-k})\times D^k \\
&=& \frac{1}{d^2}\sum_{j=n-k}^{n+e}\pi_{n*}D^{2n-k-j}\times \pi_{n*}D^j\times D^k\\
&=& 0.
\end{eqnarray*}
For every $i\geq e$, one gets that
\begin{eqnarray*}
  \pi_{n}\circ\Delta_{123}\circ(\pi_n\otimes\pi_{2i}) &=& (\pi_n\otimes\pi_{2n-2i}\otimes\pi_n)_*\Delta_{123} \\
   &=& (\pi_n\otimes\Delta_X\otimes\pi_n)_*(\frac{1}{d}p_{13}^*\delta_{X*}D^i\cdot p_2^*D^{n-i})\\
   &=& \frac{1}{d^2}\sum_{j=i}^{n+e}\pi_{n*}D^{n+i-j}\times D^{n-i}\times\pi_{n*}D^j\\
   &=& 0.
\end{eqnarray*}
By symmetry, one has 
\begin{equation*}
\pi_{n}\circ\Delta_{123}\circ(\pi_{2i}\otimes\pi_n)=0
\end{equation*}
for every $i\geq e$.

We then conclude that
\begin{eqnarray}
   &  &\Delta_{123}-\sum_{i}\sum_j\pi_{i+j}\circ\Delta_{123}\circ(\pi_i\otimes\pi_j)\label{49}\\
   &=& \sum_{i=1}^{e-1}[\pi_{n}\circ\Delta_{123}\circ(\pi_{2i}\otimes\pi_n)
   +\pi_{n}\circ\Delta_{123}\circ(\pi_{n}\otimes\pi_{2i})]
   +\sum_{k=n-e+1}^{n-1}\pi_{2k}\circ\Delta_{123}\circ(\pi_n\otimes\pi_n)\nonumber\\
   &  & +\pi_n\circ\Delta_{123}\circ(\pi_n\otimes\pi_n).\nonumber
\end{eqnarray}
Therefore, when $e\leq n$, one has $\mathfrak{d}(X)\leq\text{max}\{n, 2e-1\}$.
\medskip

\emph{The case $m\geq 3$.}
Suppose first that $k<\sum_{\ell=1}^{m}i_{\ell}-(mn-1)$.
If $i_{\ell}=n$ for every $\ell$, then necessarily $k=0$.
In this case,
\begin{eqnarray*}
   \pi_0\circ\Delta_I\circ(\pi_{n}\otimes\cdots\otimes\pi_{n})
   &=& (\pi_{n}\otimes\cdots\otimes\pi_{n}\otimes\pi_0)_*\Delta_I \\
   &=& (\pi_{n}\otimes\cdots\otimes\pi_{n})_*(\delta_{m*}o_X)\times X\\ 
   &=& 0.
\end{eqnarray*}
We may assume that $i_1=2j>n$ by symmetry.
Then
\begin{equation*}
\pi_k\circ\Delta_I\circ(\pi_{i_1}\otimes\cdots\otimes\pi_{i_m})
=\frac{1}{d}D^{n-j}\times(\pi_{2n-i_2}\otimes\cdots\otimes\pi_{2n-i_m}\otimes\pi_k)_*(\delta_{m*}D^j).
\end{equation*}
So, it suffices to show that for every $m\geq 1$, one has
\begin{equation}\label{50}
(\pi_{2n-i_2}\otimes\cdots\otimes\pi_{2n-i_m}\otimes\pi_k)_*(\delta_{m*}D^j)=0.
\end{equation}

We prove this by induction on $m$.
When $m=1$, the assertion follows from $\pi_{k*}h^j=0$, which is true, since $k\leq j$.
Now assume that $m\geq 2$. One has 
\begin{equation*}
\sum_{\ell=2}^{m}2n-i_{\ell}+k<(m-1)n+j+1.
\end{equation*}
If $i_{\ell}=n$ for every $\ell\geq 2$ and $k=n$, then $j=n$.
In this case, 
\begin{equation*}
(\pi_n\otimes\cdots\otimes\pi_n)_*((\delta_m)_*o_X)=0.
\end{equation*}
If $i_{\ell'}\neq n$ for some $\ell'$, we may assume that $i_2=2j'<n$.
Indeed, when $2j'>n$, one has $j+j'>n$ and hence $D^{j+j'}=0$.
By induction, we deduce that
\begin{eqnarray*}
   & & (\pi_{2n-i_2}\otimes\cdots\otimes\pi_{2n-i_m}\otimes\pi_k)_*(\delta_{m*}D^j) \\
   &=& \frac{1}{d}D^{n-j'}\times(\pi_{2n-i_3}\otimes\cdots\otimes\pi_{2n-i_m}\otimes\pi_k)_*((\delta_{m-1})_*D^{j+j'}) \\
   &=& 0,
\end{eqnarray*}
since $\sum_{\ell=3}^{m}2n-i_{\ell}+k<(m-2)n+j+j'+1$.
If $k=2k'<n$, then $j+n-k'>n$ and therefore
\begin{eqnarray*}
   & &  (\pi_{2n-i_2}\otimes\cdots\otimes\pi_{2n-i_m}\otimes\pi_k)_*(\delta_{m*}D^j) \\
   &=& (\pi_{2n-i_2}\otimes\cdots\otimes\pi_{2n-i_m})_*((\delta_{m-1})_*D^{j+n-k'}) \\
   &=& 0.
\end{eqnarray*}

Assume that $e\leq n$ and $k<\sum_{\ell=1}^{m}i_{\ell}-[(m-2)n+2e-1]$.
If $i_{\ell}=n$ for every $\ell$, then writing $k=2k'$, the above equality gives $k<2n-2e+1$.
Hence
\begin{eqnarray*}
   & &\pi_k\circ\Delta_I\circ(\overbrace{\pi_{n}\otimes\cdots\otimes\pi_{n}}^m)\\
   &=& (\pi_{n}\otimes\cdots\otimes\pi_{n}\otimes\pi_k)_*\Delta_I \\
   &=& \frac{1}{d}(\pi_{n}\otimes\cdots\otimes\pi_{n})_*(\delta_{m*}D^{n-k'})\times D^{k'}\\ 
   &=& \frac{1}{d^2}\sum_{i=n-k'}^{n+e}\pi_{n*}D^{2n-k'-i}\times
   (\overbrace{\pi_n\otimes\cdots\otimes\pi_n}^{m-1})_*((\delta_{m-1})_*D^{i})\times D^{k'}\\
   &=& \frac{1}{d^3}\sum_{i=n-k'}^{n+e}\sum_{j=i}^{n+e}\pi_{n*}D^{2n-k'-i}\times\pi_{n*}D^{n+i-j}\times
   (\overbrace{\pi_n\otimes\cdots\otimes\pi_n}^{m-2})
   ((\delta_{m-2})_*D^j)\times D^{k'}\\
   &=& 0.
\end{eqnarray*}
Indeed, if both $2n-k'-i=\frac{n}{2}$ and $n+i-j=\frac{n}{2}$, 
then the inequalities force $j>n$, and hence $D^j=0$.
If $i_{\ell'}\neq n$ for some $\ell'$, then we may assume that $i_1=2j_1$.
Then
\begin{eqnarray*}
   & &\pi_k\circ\Delta_I\circ(\pi_{i_1}\otimes\cdots\otimes\pi_{i_m})\\
   &=& (\pi_{2n-2j_1}\otimes\cdots\otimes\pi_{2n-i_m}\otimes\pi_k)_*\Delta_I \\
   &=& \frac{1}{d}D^{n-j_1}\times (\pi_{2n-i_2}\otimes\cdots\otimes\pi_{2n-i_m}\otimes\pi_{k})_*(\delta_{m*}D^{j_1}).
\end{eqnarray*}
It suffices to show that 
\begin{equation*}
(\pi_{2n-i_2}\otimes\cdots\otimes\pi_{2n-i_m}\otimes\pi_{k})_*(\delta_{m*}D^{j_1})=0.
\end{equation*}

We again argue by induction on $m\geq2$.
The case $m=2$ follows from (i).
Assume that $m\geq 3$.
If $i_{\ell}=n$ for every $\ell\geq 2$ and $k=n$, 
then the assumed inequality implies that $j_1\geq e$, and the required vanishing follows from the preceding argument.
Otherwise, after permuting the factors, we may assume that $i_2=2j_2\neq n$. 
Since 
\begin{equation*}
\sum_{\ell=3}^{m}2n-i_{\ell}+k<(m-2)n+2(j_1+j_2)-2e+1,
\end{equation*}
then the induction hypothesis gives
\begin{eqnarray*}
   & &(\pi_{2n-i_2}\otimes\cdots\otimes\pi_{2n-i_m}\otimes\pi_{k})_*(\delta_{m*}D^{j_1})\\
   &=& \frac{1}{d}D^{n-j_2}\times (\pi_{2n-i_3}\otimes\cdots\otimes\pi_{2n-i_m}\otimes\pi_{k})_*((\delta_{m-1})_*D^{j_1+j_2}) \\
   &=& 0.
\end{eqnarray*}

Finally, observe that if $\delta_{X*}D^k$ is completely decomposable for some $k\geq1$, then $\delta_{X*}D^i$ is completely decomposable for every $i\geq k$. The final two assertions now follow from the preceding decomposition analysis.
\end{proof}

\begin{remark}
\leavevmode
\begin{enumerate}
\item[(i)] When $e\leq n$, the proof relies on the formula \eqref{48}. This formula may fail for $i<e$, particularly when $X$ is of general type. For example, if $X$ is a very general complete intersection curve of genus $g\geq3$, the distinguished degree-one 0-cycle class $o_X$ need not be represented by a point, and hence
$\delta_{X*}o_X\neq o_X\times o_X$.

\item[(ii)] When $e=1$, namely, $X$ is a hypersurface, the above argument implies that the obstruction to multiplicativity of the natural CK decomposition is exactly the cycle class 
\begin{equation*}
\pi_n\circ\Delta_{123}\circ(\pi_n\otimes\pi_n)=(\pi_n\otimes\pi_n\otimes\pi_n)_*\Delta_{123}\in\mathrm{CH}_n^{2n}(X^3).
\end{equation*}
Consequently, if $X$ does not admit motivic $0$-multiplicativity, then $\mathfrak{d}(X)=n$.
Thus, in this setting, there is no intermediate value of the motivic multiplicativity defect.

\item[(iii)] The same argument applies to complete intersections with at most quotient singularities.

\item[(iv)] It is reasonable to expect that the components of $\Gamma^3(X,o_X)$
can be detected by their Mumford infinitesimal invariants, or more generally by the higher Abel--Jacobi maps introduced in \cite{As00,S01}.
\end{enumerate}
\end{remark}

\begin{definition}
Let $r\geq 0$ and $k\geq 1$ be integers. 
The \emph{tautological} subgroup $R^r(X^k)$ of $\mathrm{CH}^r(X^k)$ with respect to the closed embedding 
$\iota^{\times k}: X^k\rightarrow Y^k$ is defined by
\begin{equation*}
R^r(X^k):=\mathrm{Im}\left((\iota^{\times k})^*: \mathrm{CH}^r(Y^k)\rightarrow\mathrm{CH}^r(X^k)\right).
\end{equation*}
The tautological subring of $\mathrm{CH}^*(X^k)$ is $R^*(X^k):=\bigoplus_{r}R^r(X^k)$.

Let $T^*(X^k)\subseteq\mathrm{CH}^*(X^k)$ be the subring generated by $R^*(X^k)$ and the big diagonal classes $p_{ij}^*\Delta_X$, $i\neq j$, where $p_{ij}: X^k\rightarrow X^2$ is the natural projection onto the $i$-th and $j$-th factors.
\end{definition}
\begin{lemma}\label{lem4.8}
Let $Y$ be a smooth connected projective variety with trivial Chow groups, and let $X\subseteq Y$ be a smooth connected ample subvariety. Then, for every integer $r\leq\frac{n}{2}$, the restricted cycle class map
\begin{equation*}
cl_X: R^r(X)\rightarrow H^{2r}(X,\mathbb{Q})
\end{equation*}
is injective.
\end{lemma}
\begin{proof}
By the generalized Lefschetz hyperplane theorem, specifically Corollary 5.2 of \cite{Ot12}, the restriction map $\iota^*: H^{2r}(Y,{\mathbb{Q}})\rightarrow H^{2r}(X,{\mathbb{Q}})$ is injective whenever $r\leq\frac{n}{2}$.
Since $Y$ has trivial Chow groups, the cycle class map
$\mathrm{CH}^r(Y)\longrightarrow H^{2r}(Y,\mathbb{Q})$ is an isomorphism. 
It follows that the restricted cycle class map $cl_X: R^r(X)\rightarrow H^{2r}(X,\mathbb{Q})$ is injective.
\end{proof}

This naturally leads to the following question.
\begin{question}\label{ques4.9}
Let $Y$ be a smooth connected projective variety with trivial Chow groups, and let $X\subseteq Y$ be a smooth connected ample subvariety. Is the restricted cycle class map
\begin{equation*}
cl_X: R^*(X)\longrightarrow H^{*}(X,\mathbb{Q})
\end{equation*}
injective?
\end{question}
When $X$ is a smooth complete intersection deforming in the corresponding universal family, Question \ref{ques4.9} is equivalent to asking whether this family satisfies the Franchetta property. If $Y$ is a smooth weighted projective space, then Question \ref{ques4.9} has an affirmative answer.

Motivic $0$-multiplicativity can also be characterized in terms of the tautological ring introduced above.
\begin{proposition}\label{prop4.10}
Let $Y$ be a smooth connected projective variety with trivial Chow groups, and let $X\subseteq Y$ be a smooth connected ample subvariety.
Then the natural CK decomposition of $X$ is multiplicative if and only if the restricted cycle class map
\begin{equation*}
cl_{X^3}: T^*(X^3)\longrightarrow H^*(X^3,\mathbb{Q})
\end{equation*}
is injective.
\end{proposition}
\begin{proof}
By Lemma \ref{lem4.1}, for every $k\neq i+j$, the cycle class $\pi_k\circ\Delta_{123}\circ(\pi_i\otimes\pi_j)$ belongs to $T^*(X^3)$ and is homologous to zero. 
Now suppose that the restricted cycle class map $cl_{X^3}: T^*(X^3)\longrightarrow H^*(X^3,\mathbb{Q})$
is injective. Then one gets $\pi_k\circ\Delta_{123}\circ(\pi_i\otimes\pi_j)=0$.
Thus, by definition, the natural CK decomposition of $X$ is multiplicative.

For the converse, the assumption implies that the correspondence $\Gamma_{\iota}$ is of pure grade 0,
that is, $\Gamma_{\iota}=(\iota\times 1_Y)^*\Delta_Y\in\mathrm{CH}_0^{n+1}(X\times Y)$.
Consequently, for every integer $i$, the composite of morphisms of Chow motives $\mathfrak{h}^i(Y)\hookrightarrow\mathfrak{h}(Y)\stackrel{\iota^*}\longrightarrow\mathfrak{h}(X)$
factors through $\mathfrak{h}^i(X)$.
It follows that, for every integer $k\geq 1$ and every integer $j$, the composite of morphisms of Chow motives $\mathfrak{h}^j(Y^k)\hookrightarrow\mathfrak{h}(Y^k)\stackrel{(\iota^{\times k})^*}\longrightarrow\mathfrak{h}(X^k)$
factors through $\mathfrak{h}^j(X^k)$, since $(\iota^{\times k})^*=(\iota^*)^{\otimes k}$.
Hence, the restricted cycle class map $cl_{X^k}: R^*(X^k)\rightarrow H^*(X^k,\mathbb{Q})$ is injective for every integer $k\geq 1$.

On the other hand, by assumption, for every integer $0<i<\frac{n}{2}$, one has
\begin{equation}\label{51}
0=(\pi_n\otimes\pi_n\otimes\pi_{2n-2i})_*\Delta_{123}
=\sum_{j}\left(\delta_{X*}\alpha_{ij}-(\iota\times\iota)^*\omega_{ij}\right)\times\alpha_{n-i,j},
\end{equation}
where
\begin{eqnarray*}
  \omega_{ij} &:=& \sum_{k<\frac{n}{2}}\sum_{l=1}^{b_{2k}}w_{n-k, l}\times (\beta_{kl}\cdot\beta_{ij})+
\sum_{l=1}^{b_n}\gamma_{\frac{n}{2}l}\times(\gamma_{\frac{n}{2}l}\cdot\beta_{ij})+
\sum_{k>\frac{n}{2}}\sum_{l=1}^{b_{2k+2e}}\beta_{n-k, l}\times(w_{kl}\cdot\beta_{ij}).
\end{eqnarray*}
For every $0<i<\frac{n}{2}$ and every $j$, intersecting both sides of \eqref{51} with the cycle class $X^2\times\alpha_{ij}$ and then applying the push-forward $p_{12*}$, one obtains that
\begin{equation*}
\Delta_X\cdot (X\times\alpha_{ij})=\delta_{X*}\alpha_{ij}=(\iota\times\iota)^*\omega.
\end{equation*}
It follows that $R^r(X^2)=T^r(X^2)$ for every integer $r\neq n$.

The only remaining potential relation in $T^*(X^3)$ is precisely the relation expressing the multiplicativity of the natural CK decomposition. Since this relation holds by assumption, the restricted cycle class map
$cl_{X^3}: T^*(X^3)\longrightarrow H^*(X^3,\mathbb{Q})$ is injective.
\end{proof}

Based on the previous preparations, we can state the second main result of this section,
which provides a criterion for motivic 0-multiplicativity.
\begin{theorem}\label{thm4.11}
Let $Y$ and $X$ be as in Section 3, and assume that $Y$ has trivial Chow groups.
Then the natural CK decomposition of $X$ is multiplicative if and only if the following two conditions hold:
\begin{enumerate}
\item[(i)] Question \ref{ques4.9} has an affirmative answer for $X$;

\item[(ii)] $R^r(X^2)=T^r(X^2)$ for every integer $r\neq n$.
\end{enumerate}

In particular, every Fano or Calabi-Yau hypersurface in a smooth weighted projective space
admits motivic 0-multiplicativity.
\end{theorem}
\begin{proof}
By Proposition \ref{prop4.10}, conditions (i) and (ii) are necessary.

For the converse, assume conditions (i) and (ii) hold.
By Corollary \ref{cor3.6}, it follows immediately that the homologically trivial cycle class
\begin{equation*}
\Delta_{123}-\sum_{i}\sum_{j}\pi_{i+j}\circ\Delta_{123}\circ(\pi_i\otimes\pi_j)
\end{equation*}
must belong to the tautological subgroup $R^{2n}(X^3)$.
Then condition (i) gives
\begin{equation*}
\Delta_{123}=\sum_{i}\sum_{j}\pi_{i+j}\circ\Delta_{123}\circ(\pi_i\otimes\pi_j).
\end{equation*}
Hence, the natural CK decomposition of $X$ is multiplicative.
\end{proof}

\begin{remark}\label{rem4.12}
\leavevmode
\begin{enumerate}
\item[(i)] Minor modifications of the argument show that the result remains valid for any Fano or Calabi--Yau hypersurface with at worst quotient singularities in a quotient of a projective space by a finite group. For instance, a general Eisenbud--Popescu--Walter (EPW) sextic fourfold admits motivic $0$-multiplicativity.

\item[(ii)] If $Y$ is an odd-dimensional quadric hypersurface, then $\mathrm{CH}^*(Y)$ is generated by the hyperplane class, and $X$ satisfies the two conditions of Theorem \ref{thm4.11}. Thus, $X$ admits motivic $0$-multiplicativity.

\item[(iii)] Condition (ii) of Theorem \ref{thm4.11} can be deduced from the Franchetta property for the second relative power of the corresponding universal family. It does not appear easy to verify condition (ii) when $Y$ possesses non-Lefschetz cycle classes, unless such classes can be represented by subvarieties whose desingularizations have trivial Chow groups.
\end{enumerate}
\end{remark}

As direct applications of Theorem \ref{thm4.11}, we consider several interesting classes of Fano varieties.
\begin{example}
Let $X$ be a prime Fano threefold of genus $g=3$ or $5$. Then $X$ admits motivic $0$-multiplicativity.

Indeed, if $g=3$, then $X\subseteq\mathbb{P}^4$ is a quartic hypersurface, and the result follows from Theorem \ref{thm4.11}. If $g=5$, then $X$ is a complete intersection in a $5$-dimensional quadric hypersurface. The conclusion then follows from Remark \ref{rem4.12}(ii).

In a companion work \cite{X26a}, we will treat the case of all Fano threefolds.
\end{example}

\begin{example}
Let $X$ be a smooth Gushel--Mukai (GM) variety of dimension $n=3$ or $5$.
Then $X$ admits motivic 0-multiplicativity.

A generically defined self-dual CK decomposition of $X$ was explicitly constructed in \cite{FM24}. 
For $n=3$, we show that this decomposition is multiplicative. By specialization, it suffices to consider ordinary GM varieties. In this case, $X$ can be realized as a smooth, dimensionally transverse intersection
\begin{equation*}
X=G(2,5)\cap\mathbb{P}(W)\cap Q\subseteq\mathbb{P}^9,
\end{equation*}
where $\mathbb{P}(W)$ is a linear subspace of dimension $7$ and $Q$ is a quadric hypersurface. Let $Y:=G(2,5)\cap\mathbb{P}(W)$. Then $X$ is a smooth ample divisor in $Y$.
Moreover, $Y$ is Fano and has trivial Chow groups.
It is readily verified that the restriction of the cycle class map $cl: R^*(X)\rightarrow H^*(X,\mathbb{Q})$ is injective.
It remains to prove that $R^r(X^2)=T^r(X^2)$ for every $r\neq 3$,
which follows from that $\mathrm{CH}^1(X)=\mathbb{Q}\cdot H$ for each $n$ and that the cycle class $\delta_{X*}H$ is completely decomposable. The conclusion now follows from Theorem \ref{thm4.11}.

If $n=5$, then
\begin{equation*}
X=G(2,5)\cap Q,\quad Y=G(2,5).
\end{equation*}
This case was proved in \cite{La21}.

We learned through private communication that the case $n=4$ has already been treated in \cite{La26}.

The case $n=6$ will be treated in another companion work \cite{X26b}.
\end{example}

\begin{example}
Let $Y$ be a rationally connected variety of dimension $5$ with trivial Chow groups, and let $X\subseteq Y$ be a Fano hypersurface of cohomological K3 type.
Suppose that $\mathrm{CH}^*(Y)$ is generated by the hyperplane class. Then $X$ admits motivic 0-multiplicativity.

By assumption, the Néron--Severi Lie algebra of $X$ satisfies
$\mathfrak{g}_{\mathrm{NS}}(X)\cong\mathfrak{sl}_2(\mathbb{Q})$.
Since $\delta_{X*}D$ is completely decomposable, the action of $\mathfrak{g}_{\mathrm{NS}}(X)$ can be lifted to the Chow motive $\mathfrak{h}(X)$ and hence to $\mathrm{CH}^*(X)$. It follows that $R^*(X)$ is an irreducible $\mathfrak{g}_{\mathrm{NS}}(X)$-module. 
Since the cycle class map $\text{cl}_X: R^*(X)\rightarrow H^*(X,\mathbb{Q})$ is a homomorphism of $\mathfrak{g}_{\mathrm{NS}}(X)$-modules, then it is injective.
Moreover, it is clear that $R^r(X^2)=T^r(X^2)$ for every $r\neq 4$.
The conclusion therefore follows from Theorem \ref{thm4.11}.
\end{example}
\bigskip

We next determine the possible values of the motivic $2$-fold multiplicativity defect for Fano and Calabi--Yau complete intersections in smooth weighted projective spaces. The desired conclusion cannot be obtained directly from the results for generalized hypersurfaces, since the natural ambient varieties may have nontrivial Chow groups, which creates a substantial difficulty. To overcome this issue, the key observation is that the necessary constructions can be carried out relatively over the relevant parameter spaces, while the Chow rings of suitable compactifications of the resulting total spaces remain sufficiently simple for our purposes.

Let $Z=\mathbb{P}_w^N$ be a smooth weighted projective space of dimension $N=n+e$, where $e\geq 2$, 
and let $X$ be a smooth Fano or Calabi-Yau complete intersection in $Z$ of multi-degree $(d_1, d_2,\cdots, d_e)$, $d_{i+1}\geq d_i\geq 2$.
Since $X$ is Fano or Calabi--Yau, one has $\sum_{i=1}^{e}d_i\leq i_{\mathbb{P}_w^N}$, where $i_{\mathbb{P}_w^N}$ denotes the Fano index of $\mathbb{P}_w^N$.
Let $\mathbb{P}:=\mathbb{P}(\oplus_{i=1}^eH^0(\mathbb{P}_w^{N},\mathcal{O}_{\mathbb{P}_w^N}(d_i)))$ 
and $\mathbb{P}':=\mathbb{P}(\oplus_{j=1}^{e-1}H^0(\mathbb{P}_w^{N},\mathcal{O}_{\mathbb{P}_w^N}(d_j)))$) be the parameter spaces of complete intersections of multi-degrees $(d_1, d_2,\cdots, d_e)$ 
and $(d_1, d_2,\cdots, d_{e-1})$), respectively.
By specialization, we may assume that $X$ is general and write $X=V(\sigma_1,\cdots,\sigma_e)$.
After fixing $\sigma_e$, we may identify $\mathbb{P}'$ with a linear section of $\mathbb{P}$.
Let $B\subseteq\mathbb{P}$ be the locus parameterizing smooth members, and set $B':=B\cap\mathbb{P}'$. 
Then $B'$ is a nonempty open subset of the locus parameterizing smooth complete intersections of multi-degree $(d_1, d_2,\cdots, d_{e-1})$. Let $f: \mathcal{X}\rightarrow B$ and $g: \mathcal{Y}\rightarrow B'$ be the corresponding universal families of smooth complete intersections.
Let $f': \mathcal{X}'\rightarrow B'$ be the base change of $f$. Thus, $X=\mathcal{X}'_{b'_0}$ for some closed point $b'_0\in B'$.
There are natural closed embeddings over $B'$: $\mathcal{X}'\hookrightarrow\mathcal{Y}\hookrightarrow\mathbb{P}_{B'}^N$.

We retain the notation introduced in Section 3 and set $i_X:=n+e+1-\sum_{i=1}^ed_i$.
Thus, $i_X=0$ when $X$ is Calabi-Yau, whereas $i_X>0$ when $X$ is Fano.
Let $F(\mathcal{Y}/B')$ be the relative Fano scheme of relative lines of $\mathcal{Y}/B'$.
It is smooth projective over $B'$ and has dimension $n+i_X+\text{dim}B'$.
Define three varieties:
\begin{equation*}
\Theta_{\alpha,B'}=\{(x_{b'}, y_{b'}, L_{b'})\in\mathcal{X}_{/B'}^{'2}\times_{B'}F(\mathcal{Y}/B')|[\mathcal{X}'_{b'}\cap L_{b'}]=m_1x_{b'}+m_2y_{b'}, \ \text{or}\ L_{b'}\subseteq\mathcal{X}'_{b'}, \forall b'\in B'\},
\end{equation*}
\begin{equation*}
\Lambda_{\alpha, B'}:=\{L_{B'}\in F(\mathcal{Y}/B')|[\mathcal{X}'_{b'}\cap L_{b'}]=m_1x_{b'}+m_2y_{b'}, \ \text{or}\ L_{b'}\subseteq\mathcal{X}'_{b'},\ \forall b'\in B'\}
\end{equation*}
and
\begin{equation*}
\Sigma_{\alpha,B'}=\{(x_{b'}, y_{b'})\in\mathcal{X}'\times_{B'}\mathcal{X}'|[\mathcal{X}'_{b'}\cap L_{b'}]=m_1x_{b'}+m_2y_{b'}, \ L_{b'}\subseteq\mathcal{X}'_{b'}, \forall b'\in B'\}.
\end{equation*}
Let $q_{i,B'}:\Theta_{\alpha,B'}\rightarrow\mathcal{X}'$, $q_{12,B'}:\Theta_{\alpha,B'}\rightarrow\mathcal{X}'\times_{B'}\mathcal{X}'$, 
$q_{3,B'}:\Theta_{\alpha,B'}\rightarrow F(\mathcal{Y}/B')$ and $p_{\alpha,i,B'}:\Sigma_{\alpha, B'}\rightarrow\mathcal{X}'$ be the natural projections.

\begin{lemma}
The varieties $\Theta_{\alpha,B'}$, $\Lambda_{\alpha, B'}$ and $\Sigma_{\alpha,B'}$ are all of pure dimension $n+i_X+\text{dim}B'$.
\end{lemma}
\begin{proof}
The natural projection $\Theta_{\alpha, B'}\rightarrow B'$ is surjective, and its fiber over a closed point $b'\in B'$ is $\Theta_{\alpha, b'}$. By Proposition \ref{prop3.2}, the variety $\Theta_{\alpha, b'}$ is of pure dimension $n+i_X$.
It follows that $\Theta_{\alpha, B'}$ is of pure dimension $n+i_X+\text{dim}B'$.
The same argument applies also to $\Lambda_{\alpha, B'}$ and $\Sigma_{\alpha,B'}$.
\end{proof}

If $X$ is Calabi-Yau, let $\widetilde{\Theta}_{\alpha, B'}$ be any irreducible component of $\Theta_{\alpha, B'}$.
Let $\eta'$ denote the generic point of $\mathcal{X}'$.
If $X$ is Fano, then both $p_{\alpha,1,B'}^{-1}(\eta')=p_{\alpha,2,B'}^{-1}(\eta')$ and $q_{1,B'}^{-1}(\eta')=q_{2,B'}^{-1}(\eta')$ are of pure dimension $i_X$.
Moreover, the natural projection $q_{1,B'}^{-1}(\eta')\rightarrow p_{\alpha,1,B'}^{-1}(\eta')$ is generically finite and surjective. Choose a closed point $\xi'\in q_{1,B'}^{-1}(\eta')$ such that the Zariski closure of $q_{12,B'}(\xi')$ is not equal to the relative diagonal $\Delta_{\mathcal{X}'/B'}$.
Let $\widetilde{\Theta}_{\alpha, B'}$ be the Zariski closure of $\xi'$,
let $\widetilde{\Lambda}_{\alpha,B'}$ be the Zariski closure of $q_{3,B'}(\xi')$ in $F(\mathcal{Y}/B')$, and let $\widetilde{\Sigma}_{\alpha,B'}$ be the Zariski closure of $q_{12,B'}(\xi')$ in $\Sigma_{\alpha,B'}$.
Thus, $\widetilde{\Sigma}_{\alpha,B'}\neq\Delta_{\mathcal{X}'/B'}$.
The varieties $\widetilde{\Theta}_{\alpha, B'}$, $\widetilde{\Lambda}_{\alpha,B'}$ and $\widetilde{\Sigma}_{\alpha,B'}$
are integral and have dimension $n+\text{dim}B'$.
Define two further varieties:
\begin{equation*}
\widetilde{\Pi}_{\alpha, B'}:=\{(x_{b'}, y_{b'})\in\mathcal{Y}\times_{B'}\mathcal{Y}|x_{b'},y_{b'}\in L_{B'}, L_{B'}\in\widetilde{\Lambda}_{\alpha, B'}\}
\end{equation*}
and
\begin{equation*}
\widetilde{\Xi}_{\alpha, B'}:=\{(x_{b'}, y_{b'}, z_{b'})\in\mathcal{Y}\times_{B'}\mathcal{Y}\times_{B'}\mathcal{Y}|x_{b'},y_{b'},z_{b'}\in L_{B'}, L_{B'}\in\widetilde{\Lambda}_{\alpha, B'}\}.
\end{equation*}
These constructions provide relative versions of the varieties introduced in Section 3 for Fano and Calabi--Yau complete intersections.

The following result is immediate.
\begin{lemma}
The varieties $\widetilde{\Pi}_{\alpha, B'}$ and $\widetilde{\Xi}_{\alpha, B'}$ are integral of dimension $n+2+\text{dim}B'$ and $n+3+\text{dim}B'$, respectively.
\end{lemma}

The following lemma is a key ingredient.
\begin{lemma}\label{lem4.18}
For any closed point $b'\in B'$, the cycle classes $\widetilde{\Pi}_{\alpha,b'}$ and $\widetilde{\Xi}_{\alpha,b'}$ are generically defined; more precisely,
\begin{equation*}
\widetilde{\Pi}_{\alpha,b'}\in T^n(\mathcal{Y}_{b'}^2),\quad
\widetilde{\Xi}_{\alpha,b'}\in T^{2n}(\mathcal{Y}_{b'}^3).
\end{equation*}
\end{lemma}
\begin{proof}
As sets, one has 
\begin{equation*}
\widetilde{\Pi}_{\alpha, B'}\cap\mathcal{Y}_{b'}^2=\widetilde{\Pi}_{\alpha,b'},\quad
\widetilde{\Xi}_{\alpha, B'}\cap\mathcal{Y}_{b'}^3=\widetilde{\Xi}_{\alpha,b'}.
\end{equation*}
By intersection theory, one gets that
\begin{equation*}
\widetilde{\Pi}_{\alpha, B'}|_{\mathcal{Y}_{b'}^2}=l_1\widetilde{\Pi}_{\alpha,b'},\quad
\widetilde{\Xi}_{\alpha, B'}|_{\mathcal{Y}_{b'}^3}=l_2\widetilde{\Xi}_{\alpha,b'}
\end{equation*}
for some positive integers $l_1$ and $l_2$.
A natural compactification of $\mathcal{Y}$ has sufficiently simple Chow groups. Therefore, the same arguments as those used in Proposition 4.1 and Proposition 5.3 of \cite{FLV19} imply that
\begin{equation*}
\widetilde{\Pi}_{\alpha,b'}\in T^n(\mathcal{Y}_{b'}^2),\quad
\widetilde{\Xi}_{\alpha,b'}\in T^{2n}(\mathcal{Y}_{b'}^3).
\end{equation*}
\end{proof}

We can now state the third main result of this section.
\begin{theorem}\label{thm4.19}
Let $Y=\mathbb{P}_w^{N}$ be a smooth weighted projective space of dimension $N=n+e$, where $e\geq 1$,
and let $X\subseteq Y$ be a smooth Fano or Calabi--Yau complete intersection of dimension $n$. 
\begin{enumerate}
\item[(i)] There is an equality of cycle classes:
\begin{equation}\label{52}
\Delta_{123}^X=\delta_{12*}P(D_1, D_2)+\delta_{13*}P(D_1, D_2)+\delta_{23*}P(D_1, D_2)
+Q(D'_1, D'_2, D'_3)
\end{equation}
in $\mathrm{CH}^{2n}(X^3)$, where $P(s_1,s_2)\in\mathbb{Q}[s_1,s_2]$ and $Q(t_1,t_2,t_3)\in\mathbb{Q}[t_1,t_2,t_3]$ are symmetric homogenous polynomials. 
Here, $D_i=p_i^*D$ and $D'_j=q_j^*D$, where $p_i$ and $q_j$ are the natural projections.

\item[(ii)] The natural CK decomposition of $X$ is multiplicative 
if and only if the cycle class $\delta_{X*}D=\frac{1}{d_e}(\iota\times\iota)^*\Delta_Y$ is completely decomposable.
Equivalently, the correspondence $\Gamma_{\iota}$ is of pure grade 0.

\item[(iii)] The motivic 2-fold multiplicativity defect of the natural CK decomposition of $X$ takes values in $\{2k|0\leq k\leq n-1,k\neq\frac{n}{2}\}$. More precisely, it is equal to $2k$ if and only if the cycle class $\delta_{X*}D^{k+1}$ is completely decomposable, whereas $\delta_{X*}D^{k}$ is not.
\end{enumerate}
\end{theorem}
\begin{proof}
(i) We will prove the formula \eqref{52} by induction on $e\geq 1$.
The $e=1$ case is a consequence of Theorem \ref{thm4.11}.
Assume $e\geq 2$. Note that $\widetilde{\Pi}_{\alpha}=\widetilde{\Pi}_{\alpha,b'}$ for some closed point $b'\in B'$.

Denote by $H\in\mathrm{CH}^1(\mathcal{Y}_{b'})$ the hyperplane class and $D=\iota^*H$.
Since $\widetilde{\Pi}_{\alpha}$ and $\widetilde{\Xi}_{\alpha}$ are symmetric by Lemma \ref{lem4.18}, 
one gets that
\begin{equation}\label{53}
\widetilde{\Pi}_{\alpha}=\sum_{i}a_iH^{n-i}\times H^i\in T^n(\mathcal{Y}_{b'}^2)
\end{equation}
and
\begin{equation}\label{54}
\widetilde{\Xi}_{\alpha}=a(p_{12}^*\Delta_Y\cdot p_{3}^*H^{n-1}+p_{13}^*\Delta_Y\cdot p_{2}^*H^{n-1}+p_{23}^*\Delta_Y\cdot p_{1}^*H^{n-1})
+\sum_{i, j}b_{ij}H^{2n-i-j}\times H^i\times H^j
\end{equation}
in $T^{2n}(\mathcal{Y}_{b'}^3)$ for some $a, a_i, b_{ij}\in\mathbb{Q}$.
Resorting to Corollary \ref{cor3.6}, one obtains that
\begin{equation*}
\Delta_{123}=\delta_{12*}P(D_1,D_2)+\delta_{13*}P(D_1,D_2)+\delta_{23*}P(D_1,D_2)+Q(D'_1,D'_2,D'_3)
\end{equation*}
in $\mathrm{CH}^{2n}(X^n)$, where $P(s_1,s_2)\in\mathbb{Q}[s_1,s_2]$ and $Q(t_1,t_2,t_3)\in\mathbb{Q}[t_1,t_2,t_3]$ are symmetric homogenous polynomials.

(ii) This follows immediately from Theorem \ref{thm4.11}.

(iii) By part (i), the natural CK decomposition of $X$ is not $n$-multiplicative.
Since $X$ is Fano or Calabi--Yau, the distinguished $0$-cycle class $o_X$ can be represented by any point lying on a rational curve. Consequently, $\delta_{X*}o_X=o_X\times o_X$.
Therefore, the natural CK decomposition of $X$ is strictly $2k$-multiplicative
if and only if
\begin{equation*}
0\neq (\pi_n\otimes\pi_n\otimes\pi_{2n-2k})_*\Delta_{123}
=\left(\delta_{X*}D^k-\frac{1}{d}\sum_{j=k}^{n+e}D^{n+k-j}\times D^j\right)\times D^{n-k}
\end{equation*}
and
\begin{equation*}
0=(\pi_n\otimes\pi_n\otimes\pi_{2n-2(k+1)})_*\Delta_{123}
=\left(\delta_{X*}D^{k+1}-\frac{1}{d}\sum_{j=k+1}^{n+e}D^{n+k+1-j}\times D^j\right)\times D^{n-k-1}.
\end{equation*}
These are equivalent to
\begin{equation*}
\delta_{X*}D^k\neq\frac{1}{d}\sum_{j=k}^{n+e}D^{n+k-j}\times D^j,\quad
\delta_{X*}D^{k+1}=\frac{1}{d}\sum_{j=k+1}^{n+e}D^{n+k+1-j}\times D^j,
\end{equation*}
which mean that $\delta_{X*}D^{k+1}$ is completely decomposable, whereas $\delta_{X*}D^k$ is not.
\end{proof}

\begin{remark}
\leavevmode
\begin{enumerate}
\item[(i)] For a general Calabi-Yau complete intersections, part (i) implies that
the class of the cycle $\Gamma:=\bigcup_{t\in F(X)}\mathbb{P}_t^1\times\mathbb{P}_t^1\times\mathbb{P}_t^1$,
where $F(X)$ denotes the Fano variety of lines of $X$, is the restriction of a cycle class on $\mathbb{P}^N\times\mathbb{P}^N\times\mathbb{P}^N$.
This strengthens Theorem 0.7 of \cite{F13}.

\item[(ii)] The result remains valid for Fano or Calabi--Yau complete intersections with at worst quotient singularities in any quotient of a projective space by a finite group.

\item[(iii)] Unfortunately, if each variety $Y_k:=V(\sigma_k)$ has nontrivial Chow groups, 
then the cycle class $\delta_{X*}D$ need not be completely decomposable, even when $e=2$.
This explains the formulations of parts (ii) and (iii). We now describe the mechanism underlying this phenomenon.
For simplicity, take $Y=Y_k$ and write as cycle:
\begin{equation}\label{55}
\Delta_{X}-\frac{1}{d_1^{n+1}d_2}\sum_{i=1}^{n+1}X_{n+1-i}\times X_i=\sum_{j}\mathrm{div}(f_j)
\end{equation}
in the free abelian group of cycles $Z^*(Y\times Y)$, where $X_{i}$ are codimension-$i$ linear sections of $X$, 
and $f_j\in\mathbb{C}(W_j)^*$ for integral subvarieties $W_j\subseteq Y\times Y$ of dimension $n+1$.
Since the left-hand side of \eqref{55} is symmetric under the involution exchanging the two factors of $Y\times Y$, the pairs $(W_j,f_j)$ should be chosen compatibly with this symmetry. In fact, one could write down the functions $f_j$ and the varieties $W_j$ explicitly by tracing the relevant chains of rational equivalences using intersection theory.
Since $Y$ has nontrivial Chow groups, then $W_j$ should dominate $Y$ by the natural projections.
Thus, the equation \eqref{55} may fail to hold on $X\times Y$, that is, the equality $\Gamma_{\iota}=\sum_{i=0}^{n}D^i\times H^{n+1-i}$ may fail in $\mathrm{CH}^{n+1}(X\times Y)$.
Equivalently, the correspondence $\Gamma_{\iota}$ may fail to be of pure grade $0$. It follows that the cycle class $\delta_{X*}D$ need not be completely decomposable.
\end{enumerate}
\end{remark}
\bigskip

\section{Applications}
\bigskip

In this section, we present several applications of the main results established in the preceding sections.
\begin{proposition}
Let $X$ be a smooth Fano or Calabi-Yau hypersurface of dimension $n$ in a smooth weighted projective space.
Then $X$ satisfies the sharp modified diagonal property, that is, 
$\Gamma^{k}(X, o_X)=0$ in $\mathrm{CH}_n(X^k)$ for every integer $k>n$.
\end{proposition}
\begin{proof}
This follows immediately from Proposition \ref{prop2.33} and Theorem \ref{thm4.11}.
\end{proof}

\begin{proposition}\label{prop5.2}
Let $X$ be a smooth Fano or Calabi-Yau hypersurface of dimension $n$ in a smooth weighted projective space.
Let $b_{n}^{\text{tr}}(X)$ denote the dimension of the transcendental part of $H^n(X,\mathbb{Q})$.
Then, for every positive integer $k\leq 2b_{n}^{\text{tr}}(X)+1$, the restricted cycle class map
\begin{equation*}
\text{cl}_{X^k}: T^*(X^k)\rightarrow H^*(X^k,\mathbb{Q})
\end{equation*}
is injective.
Moreover, this map is injective for every integer $k\geq 1$ if and only if the Chow motive of $X$ is finite-dimensional.
\end{proposition}
\begin{proof}
The first assertion follows from Theorem \ref{thm4.11}.
The proof of the second assertion is essentially the same as that of Proposition 4.5 in \cite{FLV21}.
\end{proof}

\begin{remark}
For every integer $k\leq 2b_{n}^{\text{tr}}(X)+1$, the action of the Néron-Severi Lie algebra $\mathfrak{g}_{\text{NS}}(X^k)$ on cohomology groups can be lifted to the Chow motive $\mathfrak{h}(X^k)$ 
and hence to the Chow group $\mathrm{CH}^*(X^k)$.
The tautological subring $R^*(X^k)$ can thus be realized as an irreducible representation of $\mathfrak{g}_{\text{NS}}(X^k)$, whereas $T^*(X^k)$ can be realized as an irreducible representation of a larger Lie algebra over $\mathbb{Q}$, which need not be semisimple.
\end{remark}

\begin{corollary}
Let $f:  \mathcal{X}\rightarrow B$ be the universal family of smooth Fano or Calabi-Yau hypersurfaces in a smooth weighted projective space.
Then, for every positive integer $k\leq 3$, the $k$-th relative power $f_k: \mathcal{X}_{/B}^k\rightarrow B$ of $f$
satisfies the Franchetta property.
\end{corollary}
\begin{proof}
For every closed point $b\in B$, one has
\begin{equation*}
\text{Im}(\iota_k^*: \mathrm{CH}^*(\mathcal{X}_{/B}^k)\rightarrow\mathrm{CH}^*(\mathcal{X}_b^k))\subseteq R^*(\mathcal{X}_b^k),
\end{equation*}
where $\iota_k\colon\mathcal{X}_b^k\hookrightarrow\mathcal{X}_{/B}^k$ is the natural inclusion.
The result therefore follows from Proposition \ref{prop5.2}.
\end{proof}

More generally, we formulate the following conjecture of Franchetta type, which, among other cases, generalizes the corresponding conjecture for hyper-Kähler varieties.

\begin{conjecture}
Let $X$ and $B$ be smooth connected varieties, and let $f: X\rightarrow B$ be a smooth projective morphism
such that a very general fiber of $f$ admits motivic 0-multiplicativity.
Assume that for a very general point $b\in B$, the monodromy-invariant sub-Hodge structure $H^*(X_b,\mathbb{Q})^{\pi_1(B,b)}$ consists entirely of Hodge classes. 
Then every relative power of $f$ satisfies the Franchetta property. 
More precisely, for every integer $N\geq 1$ and every cycle class
$\Gamma\in\mathrm{CH}^*(X_{/B}^N)$,
if the restriction $\Gamma|_{X_b^N}$ is homologically trivial for a very general point $b\in B$, then
$\Gamma|_{X_b^N}=0$ in $\mathrm{CH}^*(X_b^N)$ for every closed point $b\in B$.
\end{conjecture}

As a final application, we show that motivic $0$-multiplicativity of a very general fiber of a smooth projective family gives rise to a multiplicative decomposition isomorphism in the corresponding derived category.
\begin{proposition}
Let $X$ and $B$ be smooth varieties, and let $f: X\rightarrow B$ be a smooth projective morphism. 
Assume that a very general fiber of $f$ admits a multiplicative CK decomposition. 
Then there exists a dense Zariski open subset $U\subseteq B$, such that one has a multiplicative decomposition isomorphism
\begin{equation}\label{56}
\mathbf{R}(f|_U)_*\mathbb{Q}\cong\bigoplus_i\mathbf{R}^i(f|_U)_*\mathbb{Q}[-i]
\end{equation}
in the derived category $\mathrm{D}(U; \mathbb{Q})$ of sheaves of $\mathbb{Q}$-vector spaces on $U$.

In particular, the universal family of smooth Fano or Calabi--Yau hypersurfaces in a smooth weighted projective space satisfies this property.
\end{proposition}
\begin{proof}
Let $\eta$ denote the generic point of $B$, and let $\overline{\eta}:=\mathrm{Spec}(\overline{\mathbb{C}(\eta)})\cong\mathrm{Spec}(\mathbb{C})$ be the geometric generic point. After choosing a field embedding $\overline{\mathbb{C}(B)}\hookrightarrow\mathbb{C}$,
the geometric generic fiber $X_{\overline{\eta}}$ may be identified with a very general complex fiber of $f$.
In particular, the corresponding Chow groups may be identified: $\mathrm{CH}^*(X_{\overline{\eta}}^2)\cong\mathrm{CH}^*(X_{b}^2)$ for a very general closed point $b\in B$.
Since a very general fiber of $f$ admits a multiplicative CK decomposition, the geometric generic fiber $X_{\overline{\eta}}$ does as well. By Galois descent, the generic fiber $X_\eta$ admits a multiplicative CK decomposition.

The flat pullback maps induced by restriction to the generic fibers of families commute with proper pushforward, intersection products, and the composition of correspondences. By Lemma (1A.1) of \cite{B10}, after replacing $B$ by a dense Zariski open subset $U\subseteq B$, the morphism $f|_U: f^{-1}(U)\rightarrow U$ admits a multiplicative relative CK decomposition over $U$. Let $\mathrm{CH}\mathcal{M}(U)$ denote the category of relative Chow motives over $U$,
and let $r_U: \mathrm{CH}\mathcal{M}(U)\rightarrow \mathrm{D}(U; \mathbb{Q})$ be the realization functor. 
Applying $r_U$ to the multiplicative relative CK decomposition yields the multiplicative decomposition isomorphism \eqref{56}.
\end{proof}
\bigskip

\bigskip

{\noindent\textbf{Funding}}
\medskip

This work was supported by the Fundamental Research Funds of Shandong University
under Grant No.~11140075614046, the Future Program for Young Scholars of Shandong University under Grant No.~11140089964236, and the National Natural Science Foundation of China under Grant Nos.~11601276 and 11771426.
\bigskip
\bigskip

{\noindent\textbf{Acknowledgments}}
\medskip

The author would like to express sincere gratitude to Professors Baohua Fu and Kejian Xu for their constant encouragement and support.

\vspace{1cm}

{\small

\noindent
{\bf Ze Xu}\\ School of Mathematics, Shandong University, 27 South Shanda Road, Jinan, Shandong 250100, P. R. China\\
{\bf Email: xuze@sdu.edu.cn}


\bigskip
\bigskip
\newpage

\begin{thebibliography}{99}
\bigskip
\bigskip

\bibitem[1]{As00}
Asakura, M.,
Motives and algebraic de Rham cohomology.
\emph{CRM Proc. Lecture Notes} \textbf{24},
American Mathematical Society, Providence, RI, 2000, pp.~133--154.

\bibitem[2]{Ba19}
Bazhov, I.,
On the decomposition of the small diagonal of a K3 surface.
\emph{Adv. Geom.} \textbf{19} (2019), no.~3, 353--358.

\bibitem[3]{Bea23}
Beauville, A.; Schoen, C.,
A non-hyperelliptic curve with torsion Ceresa cycle modulo algebraic equivalence.
\emph{Int. Math. Res. Not. IMRN} \textbf{2023} (2023), no.~5, 3671--3675.

\bibitem[4]{BV04}
Beauville, A.; Voisin, C.,
On the Chow ring of a K3 surface.
\emph{J. Algebraic Geom.} \textbf{13} (2004), 417--426.

\bibitem[5]{Bei87}
Beilinson, A.,
Height pairing between algebraic cycles.
In: \emph{$K$-Theory, Arithmetic and Geometry}
(Yu.~I. Manin, ed.),
Lecture Notes in Math. \textbf{1289},
Springer-Verlag, Berlin, 1987, pp.~1--26.

\bibitem[6]{BLLS23}
Bisogno, D.; Li, W.; Litt, D.; Srinivasan, P.,
Group-theoretic Johnson classes and non-hyperelliptic curves with torsion Ceresa class.
\emph{Épijournal Géom. Algébrique} \textbf{7} (2023), Art.~8, 19~pp.

\bibitem[7]{B10}
Bloch, S., Lectures on Algebraic Cycles.
New Math. Monogr. \textbf{16},
Cambridge University Press, Cambridge, 2010, xxiv+130~pp.

\bibitem[8]{BL24}
Bolognesi, M.; Laterveer, R.,
Some motivic properties of Gushel--Mukai sixfolds.
\emph{Math. Nachr.} \textbf{297} (2024), no.~1, 246--265.

\bibitem[9]{CC83}
Ceresa, G.; Collino, A.,
Some remarks on algebraic equivalence of cycles.
\emph{Pacific J. Math.} \textbf{105} (1983), no.~2, 285--290.

\bibitem[10]{D21}
Diaz, H.~A.,
The Chow ring of a cubic hypersurface.
\emph{Int. Math. Res. Not. IMRN} \textbf{2021} (2021), no.~22, 17071--17090.

\bibitem[11]{F13}
Fu, L.,
Decomposition of small diagonals and Chow rings of hypersurfaces and Calabi--Yau complete intersections.
\emph{Adv. Math.} \textbf{244} (2013), 894--924.

\bibitem[12]{FLV19}
Fu, L.; Laterveer, R.; Vial, C.,
The generalized Franchetta conjecture for some hyper-Kähler varieties.
\emph{J. Math. Pures Appl. (9)} \textbf{130} (2019), 1--35.

\bibitem[13]{FLV21}
Fu, L.; Laterveer, R.; Vial, C.,
Multiplicative Chow--Künneth decompositions and varieties of cohomological K3 type.
\emph{Ann. Mat. Pura Appl. (4)} \textbf{200} (2021), no.~5, 2085--2126.

\bibitem[14]{FM24}
Fu, L.; Moonen, B.,
Algebraic cycles on Gushel--Mukai varieties.
\emph{Épijournal Géom. Algébrique}, special volume in honour of C.~Voisin
(2024), Art.~17.

\bibitem[15]{Fu98}
Fulton, W.,
\emph{Intersection Theory}.
Second edition,
Ergeb. Math. Grenzgeb. (3), vol.~\textbf{2},
Springer-Verlag, Berlin, 1998.

\bibitem[16]{Ja94}
Jannsen, U.,
Motivic sheaves and filtrations on Chow groups.
In: \emph{Motives},
Proc. Sympos. Pure Math. \textbf{55}, Part~1,
American Mathematical Society, Providence, RI, 1994, pp.~245--302.

\bibitem[17]{KMP07}
Kahn, B.; Murre, J.; Pedrini, C.,
On the transcendental part of the motive of a surface.
In: \emph{Algebraic Cycles and Motives}, vol.~2,
London Math. Soc. Lecture Note Ser. \textbf{344},
Cambridge University Press, Cambridge, 2007, pp.~143--202.

\bibitem[18]{K94}
Kleiman, S.~L.,
The standard conjectures.
In: \emph{Motives},
Proc. Sympos. Pure Math. \textbf{55}, Part~1,
American Mathematical Society, Providence, RI, 1994, pp.~3--20.

\bibitem[19]{La19}
Laterveer, R.,
On the Chow ring of certain hypersurfaces in a Grassmannian.
\emph{Le Matematiche (Catania)} \textbf{74} (2019), no.~1, 95--108.

\bibitem[20]{La21}
Laterveer, R.,
Algebraic cycles and Gushel--Mukai fivefolds.
\emph{J. Pure Appl. Algebra} \textbf{225} (2021), no.~5,
Paper No.~106582, 18~pp.

\bibitem[21]{La24}
Laterveer, R.,
Some more Fano threefolds with a multiplicative Chow--Künneth decomposition.
\emph{Publ. Res. Inst. Math. Sci.} \textbf{60} (2024), no.~3, 561--581.

\bibitem[22]{La26}
Laterveer, R., On the generalized Franchetta conjecture for double EPW sextics. Preprint, 2026.

\bibitem[23]{LL97}
Looijenga, E.; Lunts, V.~A.,
A Lie algebra attached to a projective variety.
\emph{Invent. Math.} \textbf{129} (1997), no.~2, 361--412.

\bibitem[24]{MY16}
Moonen, B.; Yin, Q.,
Some remarks on modified diagonals.
\emph{Commun. Contemp. Math.} \textbf{18} (2016), no.~1,
Paper No.~1550009, 16~pp.

\bibitem[25]{Mu90}
Murre, J.~P.,
On the motive of an algebraic surface.
\emph{J. Reine Angew. Math.} \textbf{409} (1990), 190--204.

\bibitem[26]{Mu93}
Murre, J.~P.,
On a conjectural filtration on the Chow groups of an algebraic variety, I and II.
\emph{Indag. Math. (N.S.)} \textbf{4} (1993), 177--188 and 189--201.

\bibitem[27]{O'G14}
O'Grady, K.~G.,
Computations with modified diagonals.
\emph{Rend. Lincei Mat. Appl.} \textbf{25} (2014), 249--274.

\bibitem[28]{Ot12}
Ottem, J.~C.,
Ample subvarieties and $q$-ample divisors.
\emph{Adv. Math.} \textbf{229} (2012), no.~5, 2868--2887.

\bibitem[29]{S01}
Saito, M.,
Arithmetic mixed sheaves.
\emph{Invent. Math.} \textbf{144} (2001), no.~3, 533--569.

\bibitem[30]{Sch94}
Scholl, A.~J.,
Classical motives.
In: \emph{Motives},
Proc. Sympos. Pure Math. \textbf{55}, Part~1,
American Mathematical Society, Providence, RI, 1994, pp.~163--187.

\bibitem[31]{SV16a}
Shen, M.; Vial, C.,
The Fourier transform for certain hyper-Kähler fourfolds.
\emph{Mem. Amer. Math. Soc.} \textbf{240} (2016), no.~1139,
vii+163~pp.

\bibitem[32]{SV16b}
Shen, M.; Vial, C.,
The motive of the Hilbert cube $X^{[3]}$.
\emph{Forum Math. Sigma} \textbf{4} (2016),
Paper No.~e30, 55~pp.

\bibitem[33]{Voi04}
Voisin, C.,
Intrinsic pseudo-volume forms and $K$-correspondences.
In: \emph{The Fano Conference},
University of Turin, Turin, 2004, pp.~761--792.

\bibitem[34]{Voi12}
Voisin, C.,
Chow rings and decomposition theorems for families of K3 surfaces and Calabi--Yau hypersurfaces.
\emph{Geom. Topol.} \textbf{16} (2012), no.~1, 433--473.

\bibitem[35]{Voi15}
Voisin, C.,
Some new results on modified diagonals.
\emph{Geom. Topol.} \textbf{19} (2015), no.~6, 3307--3343.

\bibitem[36]{XX13}
Xu, K.; Xu, Z.,
Remarks on Murre's conjecture on Chow groups.
\emph{J. K-Theory} \textbf{12} (2013), no.~1, 3--14.

\bibitem[37]{X26a}
Xu, Z.,
On motivic multiplicativity of Fano threefolds.
In preparation, 2026.

\bibitem[38]{X26b}
Xu, Z.,
On motivic multiplicativity of cyclic coverings.
In preparation, 2026.

\bibitem[39]{Y13}
Yin, Q.,
\emph{Tautological Cycles on Curves and Jacobians}.
Ph.D. thesis, Radboud University Nijmegen, 2013.
Available at http://faculty.bicmr.pku.edu.cn/~qizheng.
\end{thebibliography}
\end{document}